\documentclass[preprint]{elsarticle}

\usepackage{amsmath,amsthm,amsfonts,amssymb}
\usepackage{dsfont}
\usepackage{enumerate}
\usepackage{mathtools}
\usepackage{tikz}
\usetikzlibrary{arrows.meta}
\usetikzlibrary{decorations.pathmorphing}
\usetikzlibrary{patterns}
\usetikzlibrary{calc}

\usepackage{lineno,hyperref}
\modulolinenumbers[5]

\journal{Stochastic Processes and their Applications}

\theoremstyle{plain}
\newtheorem{teo}{Theorem}[section]
\newtheorem{prop}[teo]{Proposition}
\newtheorem{lema}[teo]{Lemma}
\newtheorem{coro}[teo]{Corollary}
\newtheorem{claim}[teo]{Claim}
\theoremstyle{remark}
\newtheorem{defi}[teo]{Definition}
\newtheorem{remark}[teo]{Remark}
\newcommand{\EE}{\ensuremath{\mathbb{E}}}
\newcommand{\NN}{\ensuremath{\mathbb{N}}}
\newcommand{\PP}{\ensuremath{\mathbb{P}}}
\newcommand{\RR}{\ensuremath{\mathbb{R}}}
\newcommand{\ZZ}{\ensuremath{\mathbb{Z}}}

\newcommand{\cA}{\ensuremath{\mathcal{A}}}
\newcommand{\cC}{\ensuremath{\mathcal{C}}}
\newcommand{\cG}{\ensuremath{\mathcal{G}}}
\newcommand{\cI}{\ensuremath{\mathcal{I}}}
\newcommand{\cL}{\ensuremath{\mathcal{L}}}
\newcommand{\cP}{\ensuremath{\mathcal{P}}}
\newcommand{\cR}{\ensuremath{\mathcal{R}}}

\newcommand{\comp}{\mathsf{c}}
\newcommand{\tB}{\smash{\tilde{B}}}
\newcommand{\sB}{\smash{\mathsf{B}}}

\newcommand{\tT}{\smash{\tilde{T}}}
\newcommand{\hP}{\smash{\hat{\PP}}}
\newcommand{\tP}{\smash{\tilde{\PP}}}
\newcommand{\I}{\mathds{1}}

\newcommand{\Poi}{\mathop{\mathrm{Poi}}\nolimits}
\newcommand{\distr}{\ensuremath{\stackrel{\scriptstyle d}{=}}}

\begin{document}

\begin{frontmatter}

\title{Renewal Contact Processes: phase transition and survival}

%% include affiliations in footnotes:
\author[usp]{Luiz Renato Fontes}
\ead{lrfontes@usp.br}

\author[epfl]{Thomas S. Mountford}
\ead{thomas.mountford@epfl.ch}

\author[usp,ufrj]{Daniel Ungaretti\texorpdfstring{\corref{mycorrespondingauthor}}{}}
\cortext[mycorrespondingauthor]{Corresponding author}
\ead{daniel@im.ufrj.br}

\author[ufrj]{Maria Eulália Vares}
\ead{eulalia@im.ufrj.br}

\address[usp]{Instituto de Matem\'atica e Estat\'\i stica,
Universidade de S\~ao Paulo, SP, Brazil.}
\address[epfl]{D\'epartement de Math\'ematiques,
1015 Lausanne, Switzerland.}
\address[ufrj]{Instituto de Matem\'atica,
Universidade Federal do Rio de Janeiro, RJ, Brazil.}

\begin{abstract}
% original abstract
%%%%%%%%
%We refine some previous results concerning the Renewal Contact Processes. We
%significantly widen the family of distributions for the interarrival times
%for which the critical value can be shown to be strictly positive. The
%result now holds for any spatial dimension $d \geq 1$ and requires only a
%moment condition slightly stronger than finite first moment. We also prove
%a Complete Convergence Theorem for heavy tailed interarrival times.
%Finally, for heavy tailed distributions we examine when the contact
%process, conditioned on survival, can be asymptotically predicted knowing
%the renewal processes. We close with an example of an interarrival time
%distribution attracted to a stable law of index 1 for which the critical
%value vanishes, a tail condition uncovered by previous results.

% shorter version: at most 100 words
We refine previous results concerning the Renewal Contact Processes. We
significantly widen the family of distributions for the interarrival times
for which the critical value can be shown to be strictly positive. The
result now holds for any dimension $d \ge 1$ and requires only a moment
condition slightly stronger than finite first moment. For heavy-tailed
interarrival times, we prove a Complete Convergence Theorem and examine
when the contact process, conditioned on survival, can be asymptotically
predicted knowing the renewal processes. We close with an example of
distribution attracted to a stable law of index 1 for which the critical
value vanishes.
\end{abstract}

\begin{keyword}
contact process\sep percolation\sep renewal process
\MSC[2020] 60K35\sep 60K05\sep 82B43
\end{keyword}

\end{frontmatter}

%\linenumbers
%\section*{To do}
%\label{sec:to_do}
%\begin{itemize}
%\item Report: 10
%\end{itemize}

\section{Introduction}
\label{sec:introduction}

In this note we address natural questions arising from the
papers~\cite{FMMV,FMV} that deal with an extension of the classical contact
process introduced by Harris in~\cite{H} as a model for the spread of a
contagious infection. The sites of $\mathbb{Z}^d$ are thought as the
individuals, the state of the population being represented by a configuration
$\xi \in \{0,1\}^{\mathbb{Z}^d}$, where $\xi(x)=0$ means that the individual
$x$ is healthy and $\xi(x)=1$ that $x$ is infected. A Markovian evolution was
then considered: infected individuals get healthy at rate 1 independently of
everything else, and healthy individuals get sick at a rate that equals a given
parameter $\lambda$ times the number of infected neighbours. Harris contact
process, as it is usually called, is one of the most studied interacting
particle systems (see e.g.~\cite{L,L2}) and has also opened a very wide road to
multiple generalizations that have distinct motivations and potential
applications, including space or time inhomogeneities, more general graphs and
random graphs. A variety of random environments may also be modelled by
considering suitable families of random rates. Results regarding
survival and extinction for contact processes with random environments can be
found in~\cite{Andjel, BDS, Lig92, NV, K}. Also, Garet and Marchand~\cite{GM}
prove a shape theorem in this context.

In~\cite{Ha78}, Harris introduced a percolation structure on which the
contact process was built, also known as graphical representation, in terms of
a system of independent Poisson point processes. This has shown to be extremely
useful not only to prove various basic properties of the process, but also for
renormalization arguments (see e.g.~\cite{DG, BG}). It was exactly this angle
that motivated the investigation started in~\cite{FMMV, FMV}, leading to the
consideration of more general percolation structures, where the Poisson times
would give place to more general point processes, so that the Markov property
is lost, but the percolation questions continue to be meaningful and pose new
challenges.

The extension of the contact process that we consider is what we call Renewal
Contact Process (RCP). It is a modification of the Harris graphical
representation in which transmissions are still given by independent Poisson
processes of rate $\lambda>0$, but cure times are given by i.i.d. renewal
processes with interarrival distribution $\mu$, a model we denote by
RCP($\mu$). For definiteness, we take the starting times of all renewal
processes to be zero, but this choice does not affect our arguments.

In this paper we improve the current understanding of survival and extinction
in RCP($\mu$) provided by~\cite{FMMV, FMV}. The critical parameter for
RCP($\mu$) is defined as
\begin{equation*}
\label{eq:lambda_c}
    \lambda_c(\mu) := \inf \{\lambda:\; P(\tau^{0} = \infty) > 0\},
\end{equation*}
where $\smash{\tau^{0} := \inf \{t:\; \xi_t^{\{0\}} \equiv 0\}}$ and
$\smash{\xi_t^{\{0\}}}$ is the process started from the configuration in
which only the origin is infected.
(As usual, we make the convention that $\inf \emptyset=\infty$.)

Reference~\cite{FMV}
considered sufficient conditions on $\mu$
to ensure that $\lambda_c(\mu) > 0$. The first contribution of the present paper
is a new construction, simpler than the one in~\cite{FMV}, that results in two
meaningful improvements.  Firstly,
the present construction works for every dimension $d \geq 1$.
Secondly, 
we significantly relax the assumptions on $\mu$, as described by the following result:
\begin{teo}
\label{teo:phase_transition_improved}
Consider a probability distribution $\mu$ satisfying
\begin{equation}
\label{eq:theta_choice_mu_improved}
\int_{1}^{\infty} x
    \exp\Bigl[ \theta (\ln x)^{1/2} \Bigr] \mu(\mathrm{d}x)
    < \infty
\quad \text{for some $\theta >  4\sqrt{d \ln 2}$}.
\end{equation}
Then, the RCP($\mu$) has $\lambda_c(\mu) > 0$. In particular, $\lambda_c(\mu) > 0$
whenever\break $\int x^{\alpha} \mu(\mathrm{d}x) < \infty$ for some $\alpha > 1$.
\end{teo}

The construction that leads to Theorem~\ref{teo:phase_transition_improved} is
presented in Section~\ref{sec:phase_transition}. Essentially, it shows that
if the probability that a renewal process $\cR$ with interarrival distribution
$\mu$ has a large gap is sufficiently small, 
then the critical parameter for the RCP is strictly positive.
The moment condition in~\eqref{eq:theta_choice_mu_improved}, together with
Lemma~\ref{lema:moment_condition}, can be seen as a quantitative control on
the probability of having large gaps.

Let us first discuss previous results that
hold for the RCP on $\ZZ^{d}$ with any spatial dimension $d \geq 1$.
Theorem~1 of~\cite{FMV} proves that {$\lambda_c(\mu) > 0$} if
$\mu$ has finite second moment. On the other hand, in~\cite{FMMV} it is
proved that if there are $\epsilon, C_1 > 0$ and $t_0 > 0$ such that
$\mu([t,\infty)) \geq C_1/t^{1-\epsilon}$ for all $t \geq t_0$, then
(under some auxiliary regularity hypothesis) $\lambda_c(\mu) = 0$.
Notice that for general dimension these previous results leave a large gap
between distributions $\mu$ for which {$\lambda_c(\mu) > 0$} has been proven
and those for which we know $\lambda_c(\mu) = 0$.

In the specific case of spatial dimension $d=1$ this gap was considerably
smaller. Theorem~2 of~\cite{FMV} proves that
{$\lambda_c(\mu) > 0$}
if $\mu$ satisfies $\int t^{\alpha} \mu(\mathrm{d}t) < \infty$ for some
$\alpha > 1$, has a density and a decreasing hazard rate.
Therefore, Theorem~\ref{teo:phase_transition_improved} represents a considerable
improvement on conditions for ${\lambda_c(\mu) > 0}$.

In the proof of Theorem~2 of~\cite{FMV}, the density and decreasing
hazard rate of $\mu$ are used to show that RCP($\mu$) satisfies
an FKG inequality, a tool repeatedly used in the proof of that theorem,
combined with a crossing property of infection paths which holds only in $d=1$.
The construction used for proving Theorem~\ref{teo:phase_transition_improved}
has a similar overall structure, with
the crucial difference that it does not require the path crossing property or
FKG, and thus allows more general distributions and dimensions.

We stress that the moment condition in~\eqref{eq:theta_choice_mu_improved}
shows that there are distributions $\mu$ on the domain of attraction of a
stable law with index 1 for which {$\lambda_c(\mu) > 0$}.
On the other hand, in Section~\ref{sec:example} we give an example
(see Theorem~\ref{teo:example1}) of a measure $\mu$ in the domain of
attraction of stable with index 1 for which the critical parameter
vanishes. One may be tempted to conjecture that $\lambda_c(\mu) > 0$ is
equivalent to $\mu$ having a finite first moment. Up until now we have not
been able to find a counter-example to this statement.

The discussion so far is concerned with sufficient conditions to ensure
that $\lambda_c(\mu)$ is zero or positive, and this is indeed one of the main
goals of this paper. Nevertheless, it is also natural to ask whether
$\lambda_c(\mu)< \infty$, so that we may speak of a {\it phase transition}.
Clearly, for a degenerate $\mu$ (e.g.  $\mu(\{1\})=1$)  the infection always
dies out (at time 1), so that $\lambda_c(\mu)=\infty$.  This pathologic
behavior should not occur once we avoid the phenomenon of simultaneous
extinction. Proving a precise mathematical result demands care, and we still do
not have a complete answer. Of course, since $\lambda_c(\mu)$ is clearly
non-increasing in $d$,  it suffices to consider the case $d=1$. A simple
sufficient condition can be given if we restrict to the class of measures $\mu$
considered in \cite{FMV}. If $\mu$ has a density and a bounded and decreasing
hazard rate, then $\lambda_c(\mu)$ is finite.  This is further explained in
Remark~\ref{finite} in the next section. When $d \ge 2$ the situation is much
simpler, and one can avoid the dependencies within each renewal process, simply
by using each of them only once to construct an infinite infection path, i.e.
through a coupling with supercritical oriented percolation, 
analogously to what was done in the proof of Theorem~1.3(ii) in~\cite{HUVV}.

The other results in the paper focus on the long time behavior of the
RCP($\mu$) for $\mu$ such that ${\lambda_c(\mu) = 0}$. 
Reference~\cite{FMMV} provides the following conditions on $\mu$ to ensure
that a RCP($\mu$) has critical value equal to zero:
\begin{enumerate}[A)]
\item There is $1 < M_1 < \infty$, $\epsilon_1>0$ and $t_1 > 0$ such that
    \begin{equation*}
    \text{for every $t > t_1$,} \qquad
        \epsilon_1 \smash{\int_{[0,t]}} s \mu(\mathrm{d}s) < t \mu(t, M_1 t).
    \end{equation*}
\item There is $1 < M_2 < \infty$, $\epsilon_2>0$ and $r_2 > 0$ such that
    \begin{equation*}
    \text{for every $r > r_2$,} \qquad
        \epsilon_2 \mu[M_2^{r}, M_2^{r+1}] \le \mu[M_2^{r+1}, M_2^{r+2}].
    \end{equation*}
\item There is $M_3 < \infty$, $\epsilon_3>0$ such that
\begin{equation*}
    \text{for $t \ge M_3$,} \qquad
        t^{-(1-\epsilon_3)} \le \mu(t,\infty) \le t^{-\epsilon_3}.
    \end{equation*}
\end{enumerate}
These conditions require that $\mu$ has a heavy, mildly regular tail.
Under them, it is shown in~\cite{FMMV} that for any infection rate
$\lambda > 0$, one can find an event of positive probability in which the
infection survives -- but see Remark~\ref{relax}, where we argue that the
upper bound in C) may be dropped as a hypothesis for this result to hold;
we will need it for the results of the present article, though (as explained
in Remark~\ref{comp}). In the event just mentioned, the path along which the
infection survives goes to ``infinity" as time diverges, so there is no
information in that event about strong survival of the process (in whichever
way this may be defined, see~\cite{Pe}). In
Section~\ref{sec:complete_convergence} we show the following result.

\begin{teo}
\label{teo:complete_convergence}
Let interarrival distribution $\mu$ satisfy conditions A)-C) of Theorem~1
of~\cite{FMMV}. Then, for a RCP starting from any initial condition $\xi_0$
we have that $\xi_t$ converges in law, as $t \to \infty$, to
\begin{equation}
\label{convergence}
P(\tau < \infty)\delta_{\underline{0}} + P(\tau = \infty)\delta_{\underline{1}},
\end{equation}
where $\tau = \inf\{t>0: \xi_t\equiv0\}$, and $\delta_{\underline{0}}$ and
$\delta_{\underline{1}}$ represent the Dirac measure on the configuration with
all sites healthy and all sites infected, respectively.
\end{teo}
Given Theorem~\ref{teo:complete_convergence}, it is natural to see the sites
(conditional upon survival of the process) as being a solid growing block of
points which lose their infection ever more rarely and are quickly reinfected
by their infected neighbours. Section~\ref{sec:closeness_determinism} develops
this picture further, under stricter regularity conditions for the tail of
$\mu$, demanding that it be attracted to an $\alpha$-stable law with
$0<\alpha<1$, with some extra regularity for $\alpha<1/2$. Given a fixed
site (e.g.~the origin), it is natural to expect that given the information
supplied by the renewal process, and in the event of survival of the infection
started at the origin, the conditional probability that $\xi_t(0)=1$ will be
close to $1-e^{-2\lambda d Y_t(0)}$, where $\cR_0$ is the renewal process at
the origin and $Y_t(0):=t-\sup\{\cR_0 \cap [0,t]\}$ is the {\em age} of $\cR_0$
at time $t$, or, in other words, the time elapsed up to time $t$ since the most
recent renewal of $\cR_0$ prior to $t$.

We will effectively confirm this expectation for $\alpha<1/2$, showing that in
this case
\begin{equation*}
\lim_{t \to \infty}
    \bigl|
    P(\xi_t(0)=1\mid \cR, \text{survival}) -
    (1-e^{-2\lambda d Y_t(0)})
    \bigr|
    =0,
\end{equation*}
see Theorem~\ref{teo:closeness}.
For $\alpha \ge 1/2$, things get more complex, and indeed we show
(in the same theorem) that
\begin{equation*}
\varlimsup_{t \to \infty}
    \Bigl(
    1-e^{-2\lambda d Y_t(0)} -  P(\xi_t(0)=1\mid \cR, \text{survival})
    \Bigr)
    >0
\end{equation*}
for $\alpha>1/2$. A more precise result is stated in
Theorem~\ref{teo:closeness_improved}.

We close this introduction with a discussion on related papers. There are
affinities between our RCP and the treatment of contact processes in a class of
random environment as in~\cite{K,NV}. The main novel aspect of RCP is the loss
of the Markov property. Similarities are also present in the renormalization
arguments used in Section~\ref{sec:phase_transition} and those in~\cite{BG}.

In~\cite{FGS}, RCP has been studied in the context of finite graphs. It
deals with the RCP($\mu$) on finite connected graphs, say of size $k$, with
$\mu$ attracted to an $\alpha$-stable law with $0<\alpha<1$.  Estimates close
to optimal are derived for the critical size of the graph at and above which
we have $\lambda_c(\mu) = 0$ (and below which $\lambda_c(\mu) = \infty$):
except for countably many such $\alpha$'s, the estimates are sharp; for the
exceptional $\alpha$'s, there is exactly one value of $k$ for which the value
of $\lambda_c$ is undetermined.  Similar ideas appear in connection with
quantum versions of the Ising model and highly anisotropic Ising
models~\cite{AKN,IL,FMMPV}.

Finally, motivated by different random environments for the contact
process, other variations of RCP($\mu$) have been considered in \cite{HUVV},
where the transmissions are also given by renewal processes.

\section{Extinction}
\label{sec:phase_transition}

\subsection{Main events}
\label{sub:main_events}

Our construction relates the probability of crossing a box in some direction
for a well-chosen sequence of boxes that we define below.
One important difference from the previous construction from~\cite{FMV} is
a crossing event which we call a \textit{temporal half-crossing}.
A general space-time crossing is defined in~\cite{FMV} as follows.
\begin{defi}[Crossing]
Given space-time regions $C, D, H \subset \ZZ^{d} \times \RR$ we say there
is a crossing from $C$ to $D$ in $H$ if there is a path
$\gamma:[s, t] \to \ZZ^{d}$ such that $(\gamma(s), s) \in C$, $(\gamma(t), t) \in D$
and for every $u \in [s,t]$ we have $(\gamma(u),u) \in H$.
\end{defi}

Given a space-time box
$B := \bigl(\prod_{i=1}^{d} [a_i, b_i]\bigr)\times [s, t]$ we
usually denote its space projection as $[a, b]$ where $a = (a_1, \ldots, a_d)$
and $b = (b_1, \ldots, b_d)$. Also, we refer to its faces at direction
$1 \le j \le d$ by
\begin{equation*}
    \partial_{j}^{-}B := \{(x, u) \in B;\; x_j = a_j\}
    \quad \text{and} \quad
    \partial_{j}B := \{(x, u) \in B;\; x_j = b_j\}.
\end{equation*}
Using this notation, we have three crossing events of box
$B = [a,b] \times [s,t]$ that are important in our investigation.
\begin{description}
\item[\textbf{Temporal crossing.}]
Event $T(B)$ in which there is a path from $[a,b]\times \{s\}$ to
$[a,b]\times \{t\}$ in $B$.

\item[\textbf{Temporal half-crossing.}]
Event $\tT(B) := T([a,b] \times [s, \frac{t+s}{2}])$.
In words, we have a temporal crossing from the bottom of $B$ to
the middle of its time interval.

\item[\textbf{Spatial crossing.}]
For some fixed direction $j \in \{1,\ldots, d\}$ we define event $S_j(B)$
in which there is a crossing from $\partial_{j}^{-}B$ to $\partial_{j}B$
in $B$, i.e., there is a crossing connecting the opposite faces of
direction $j$.
\end{description}

These events are the basis of our analysis of phase transition in RCP. Consider
sequences $a_n, b_n$ and fix a sequence of boxes
$B_n = [0, a_n]^{d} \times [0, b_n]$. We want to relate
\begin{enumerate}
\item Crossings of box $B_n$ to crossings of boxes at smaller scales.
\item Event $\{\tau^{0} = \infty\}$ to crossings of boxes at some scale $n$.
\end{enumerate}

From 1. we will obtain recurrence inequalities showing that the
probability of crossing a box of scale $n$ is very small for large $n$ and
this in turn will imply that in 2. we have $\PP(\tau^{0} = \infty) = 0$.

Considering a box $B = [-a_n/2, a_n/2]^{d} \times [0,b_n]$, we can see that
if the infection of the origin survives till time $b_n$ then either we have
$T(B)$ or the infection must leave box $B$ through some of its faces
$\partial_{j}B$ or $\partial_{j}^{-}B$ for $1 \le j \le d$.
Fix some diretion $j$ and notice that $\{(x, u) \in \ZZ^{d} \times \RR;\; x_j=0\}$
divides box $B$ into two halves. Denote by $\tB_j$ the half containing face
$\partial_{j}B$. Since the infection path is càdlàg,
if we have a path leaving $B$ through $\partial_{j}B$ then event
$S_j(\tB_j)$ occurred. Thus, by symmetry and the union bound one can write
\begin{equation}
\label{eq:survival_prob}
\PP(\tau^{0} = \infty)
    \le \PP(T(B)) + 2d \cdot \PP(S_1(\tB_1)).
\end{equation}
This quite simple relation already tells us that it suffices to prove that the
probability of temporal crossings of $B$ and spatial crossings of half-boxes
in the short direction go to zero as $n \to \infty$.

\subsection{General moment condition}
\label{sub:general_moment_condition}

We consider the sequence of space-time boxes
$B_n = [0, a_n]^{d} \times [0, b_n]$. Also, we denote by $\tB_j(n)$
the half-box of $B_n$ that contains the face $\partial_j B_n$. We are concerned
with the probability of the following events:
\begin{equation}
\label{eq:important_events}
    S_j(B_n), \quad
    T(B_n), \quad
    \tT(B_n) \quad \text{and} \quad
    S_j(\tB_j(n)).
\end{equation}
Notice that the probability of events in which some direction
$j$ appear are actually independent of $j$ by symmetry.
Another important remark is that whenever we translate a box by
$(x,0) \in \ZZ^d \times \RR$ the probability of any of these crossing events
remains the same. However, in order to disregard the specific position of our
boxes in space-time and also the possible knowledge of some renewal marks below
the box in consideration, it is useful to define the following uniform
quantities.

\begin{defi}
\label{defi:uniform_quantities}
We define
\begin{align}
\label{eq:uniform_quantities}
\begin{aligned}
s_n
    &:= \sup \smash{\hat{\PP}}(S_j((x,t) + B_n)), \\
h_n
    &:= \sup \smash{\hat{\PP}}(S_j((x,t) + \tB_j(n))),
\end{aligned}
    &&
\begin{aligned}
t_n
    &:= \sup \smash{\hat{\PP}}(T((x,t) + B_n)),
    \\
\tilde{t}_n
    &:= \sup \smash{\hat{\PP}}(\tT((x,t) + B_n)),
\end{aligned}
\end{align}
where the suprema above are over all $(x,t) \in \ZZ^{d} \times \RR_+$ and all
product renewal probability measures $\smash{\hat{\PP}}$ with interarrival
distribution $\mu$ and renewal points starting at (possibly different)
time points strictly less than zero.
\end{defi}

Notice also that the quantities in which some
direction $j$ appear are actually independent of $j$ by symmetry.
Using~\eqref{eq:survival_prob} and the uniform quantities defined
in~\eqref{eq:uniform_quantities}, we can estimate
\begin{equation*}
\PP(\tau^{0}=\infty)
    \le t_n + 2d \cdot h_n
    \le \tilde{t}_n + 2d \cdot h_n.
\end{equation*}
We just have to show the right hand side goes to zero, giving upper bounds
to the quantities $\tilde{t}_n$ and $h_n$.
This is done recursively, relating quantities from consecutive scales.
Heuristically, we prove that whenever we have a crossing
on scale $n$ we must have two `independent' crossings (either spatial crossings
or temporal half-crossings) of boxes of the previous scale that are inside
the original box.

Notice that if we are moving on a spatial direction, then this independence is
immediate. For instance, it is clear that in order to cross $B_n$ on the first
coordinate direction we must cross both $\tB_1(n)$ and $B_n \setminus
\tB_1(n)$. Since these events rely on independent processes, we have that
$s_n \le h_n^{2}$.

However, when moving on the time direction we might have dependencies; here,
the uniform quantities prove their usefulness.
The next lemma gives a uniform estimate on the probability of not having
renewal marks on an interval, making it useful to adjust our choice of
sequence $b_n$ that represents the time length of our sequence of boxes $B_n$.
\begin{lema}[Moment condition]
\label{lema:moment_condition}
Let $\mu$ be any probability distribution on $\RR_+$ and $\cR$ be a renewal
process with interarrival $\mu$ started from some $\mathfrak{t} \le 0$.
Let $f: [0, \infty) \to [0,\infty)$ be non-decreasing, differentiable
and satisfying $f(0) = 0$ and $f(x) \uparrow \infty$ as $x \to \infty$.
If $\int x f(x) \,\mu(\mathrm{d}x) < \infty$, then uniformly on
$\mathfrak{t}$ we have
\begin{equation}
\label{eq:moment_condition}
\sup_{t \geq 0} \PP( \cR \cap [t, t + u] = \varnothing)
    \le \frac{C}{f(u)},
\end{equation}
for some positive constant $C = C(\mu, f)$ whenever $f(u) > 0$.
\end{lema}

\begin{proof}
The proof is a standard application of renewal theorem. We can assume
$\mathfrak{t} = 0$ since the case $\tilde{\mathfrak{t}} < 0$ is the same as taking a supremum
over intervals $[t, t+u]$ with $t \geq -\tilde{\mathfrak{t}}$ and a renewal
started from 0.

Let us first assume $\mu$ is non-arithmetic.
Denote by $F$ the cumulative distribution function
of $\mu$ and let $\bar{F} = 1 - F$.
Moreover, denote the overshooting at $t$ for renewal $\cR$
(i.e.\ , the time till the next renewal mark after $t$) by
$Z_t$ and let ${H(t) := \EE[f(Z_t)]}$. Conditioning with respect to
the first renewal $T_1$, we have
\begin{align*}
H(t)
    &= \EE[f(T_1 - t) \I\{T_1 > t\}] +
        \EE\bigl[\I\{T_1 \le t\} \EE[f(Z_t) \mid T_1]\bigr] \\
    &= \EE[f(T_1 - t) \I\{T_1 > t\}] +
        \EE\bigl[\I\{T_1 \le t\} \EE[f(Z_{t-T_1})]\bigr] \\
    &= \int_{t}^{\infty} f(x-t) \,\mathrm{d}F(x) +
        \int_{0}^{t} H(t-x) \,\mathrm{d}F(x).
\end{align*}
Denoting the first integral above by $h(t)$, the equality above is the renewal
equation $H = h + H \ast F$. Some alternative expressions for $h(t)$ are
\begin{equation}
    \label{eq:expressions_h}
    h(t)
    = \int_{t}^{\infty} f'(x-t) \bar{F}(x) \,\mathrm{d}x
    = \int_{0}^{\infty} f'(s) \bar{F}(s+t) \,\mathrm{d}s.
\end{equation}
To justify integration by parts in this step, we write
$\int_{t}^{\infty} f(x-t) \,\mathrm{d}F(x) =\lim_{L\to\infty}\int_{t}^{L} f(x-t) \,\mathrm{d}F(x)$, 
and then perform the latter integral by parts (assuming $L>t$), obtaining 
$\int_{t}^{L} f'(x-t) \bar F(x) \,\mathrm{d}x+f(0)\bar F(t)-f(L-t)\bar F(L)$. 
Using the monotonicity of $f$ and Markov's inequality, we find that
\begin{equation*}
f(L-t)\bar F(L)\leq f(L)\bar F(L)\leq \EE[T_1f(T_1)]/L,
\end{equation*}
and the justification follows immediately from our other assumptions on $f$.

Let $X$ be a random variable with distribution $\mu$.
From~\eqref{eq:expressions_h} it is easy to see that
${h(0) = \EE f(X) < \infty}$ and that $h$ is decreasing in $t$.
Also, we can evaluate
\begin{align*}
\int_{0}^{\infty} h(t) \,\mathrm{d}t
    &= \int_{0}^{\infty} \int_{0}^{\infty} f'(s) \bar{F}(s+t)
        \,\mathrm{d}s \,\mathrm{d}t \\
    &= \int_{0}^{\infty} f'(s) \int_{0}^{\infty} \bar{F}(s+t)
        \,\mathrm{d}t \,\mathrm{d}s \\
    &\le \int_{0}^{\infty} f'(s) \EE \bigl[ X \I\{X > s\} \bigr]
        \,\mathrm{d}s \\
    &= \EE \Bigl[ X \int_{0}^{X} f'(s) \,\mathrm{d}s \Bigr] \\
    &= \EE[ X f(X)].
\end{align*}
Thus, we have that $h$ is directly Riemann integrable when
$\EE[ X f(X)] < \infty$ and the renewal theorem implies
\begin{equation*}
    H(t) = \EE[f(Z_t)] \to \frac{\EE [X f(X)]}{\EE X}
\end{equation*}
as $t \to \infty$. Separating the cases in which $t$ is large and 
$t$ is small, we have a uniform bound on $t$ for $H(t)$. 
For the latter control, notice that $H(t) = \EE[f(Z_t)]$ may be written 
as $\int_0^t U(ds)\int_{t-s}^\infty\mu(dr) f(r-(t-s))$, where $U$ is the
renewal measure associated to $\mu$; from our assumptions on $f$,
it follows that the innermost integral above is bounded by
$\int_{0}^\infty\mu(dr) f(r) = c <\infty$, and thus
$H(t) \leq  c \cdot U(0,t)$ is finite for every $t>0$.
From the fact that $U$ is nondecreasing, we get that $H(t)$ is bounded 
on bounded intervals.

Since $f$ is non-negative and non-decreasing, by Markov inequality we can write
\begin{equation*}
\PP(Z_t \geq u)
    \le \frac{\EE f(Z_t)}{f(u)}
    \le \frac{C}{f(u)}.
\end{equation*}
The conclusion in~\eqref{eq:moment_condition} follows, under the assumption
that $\mu$ is non-arithmetic. The arithmetic case is even simpler and a minor
change in the argument yields the same bound.
\end{proof}

When we know that in a box $[0,a_n]^{d} \times [s, t]$ 
every site $x \in [0,a_n]^{d}$ has a renewal mark, analyzing crossing events
on $[0,a_n]^{d} \times [t, \infty)$ gets easier since we are able
to forget all information from time interval $[0,s]$. Our next result uses
Lemma~\ref{lema:moment_condition} to estimate the probability that such
event does not occur.
\begin{coro}
\label{coro:moment_condition}
Let $\mu$ satisfy~\eqref{eq:theta_choice_mu_improved} and
$f(x) := e^{ \theta (\ln x)^{1/2} } \I_{\{x \geq 1\}}$.
Define $J_n(t,s)$ as the event in box
$[0,a_n]^{d} \times [t, t + s]$ in which there is some
site $x \in [0,a_n]^{d}$ with no renewal marks on $[t, t + s]$.
Then, for any $s \ge 2$ we have
\begin{equation*}
\sup_{t \geq 0} \PP(J_n(t,s))
    \le \frac{C a_n^{d}}{f(s)}.
\end{equation*}
\end{coro}

\begin{proof}
Modify $f$ in $[1,2]$ to ensure differentiability and use the union bound.
\end{proof}

\subsection{Relating successive scales}
\label{sub:relating_successive_scales}

In this section we prove uniform upper bounds for $\tilde{t}_n$ and $h_n$ in
terms of $h_{n-1}$ and $\tilde{t}_{n-1}$. From here on we consider boxes $B_n$
with $a_{n} = 2^{n}$.

\medskip
\noindent
\textbf{Temporal half-crossings.} Let us upper bound the quantity
$\tilde{t}_n$. For this part we work under the assumption that $\mu$
satisfies~\eqref{eq:theta_choice_mu_improved}. Define
\begin{equation*}
G_i := T\bigl([0,2^{n}]^{d} \times [ib_{n-1},(i+1)b_{n-1}]\bigr)
\end{equation*}
and notice that event $G_i$ is measurable with respect to the
$\sigma$-algebra that looks all renewal processes and Poisson
processes of $B_n$ up to time $(i+1) b_{n-1}$.
Moreover, consider event $J = J_n(b_{n-1},b_{n-1})$ defined in
Corollary~\ref{coro:moment_condition} and notice $J$ is depends on the point
processes up to time $2b_{n-1}$.

Assuming that $b_n/2 > 3 b_{n-1}$, notice that we have
\begin{equation*}
\tT(B_n)
    \subset J \cup (G_0 \cap J^{\comp} \cap G_2),
\end{equation*}
implying that we can write
\begin{equation*}
\hP(\tT(B_n))
    \le \hP(J) + \hP(G_0) \cdot \hP(G_2 \mid G_0 \cap J^{\comp}).
\end{equation*}

Corollary~\ref{coro:moment_condition} provides an upper bound for $\hP(J)$.
Moreover, we can estimate the conditional probability by integrating over
all possible collections $\{\mathfrak{t}_x; x \in [0,2^{n}]^{d}\}$ of time points in
$[b_{n-1},2b_{n-1}]$ the probability of event $G_2$. For any fixed choice
of such collection, denote by $\tP$ the probability measure with
starting renewal marks given by $(x, \mathfrak{t}_x - 2b_{n-1})$. This leads to the
bound
\begin{equation*}
\hP(\tT(B_n))
    \le \frac{C 2^{dn}}{f(b_{n-1})} +
        \hP(G_0) \cdot \sup\nolimits_{\{\mathfrak{t}_x\}} \tP(G_0).
\end{equation*}
The last product on the right hand side may be estimated by
\begin{equation*}
    \sup\nolimits_{\{\mathfrak{t}_x\}}
    \bigl(\hP(T([0,2^{n}]^{d} \times [0, b_{n-1}]))\bigr)^{2},
\end{equation*}
where in the supremum we now consider any possible starting collection of time
points $\{\mathfrak{t}_x;\; x \in \ZZ^d, \mathfrak{t}_x \le 0\}$ and we use notation $\hP$ to emphasize this.
We look for an upper bound that is valid for any
starting renewal marks. In order to bound $\hP(T([0,2^{n}]^{d} \times [0, b_{n-1}]))$,
we partition $[0,2^{n}]^{d}$ into sub-boxes of side length $2^{n-2}$.
Considering projections of our crossing into space, we can prove
\begin{lema}[Temporal half-crossing]
\label{lema:temporal_half_cross}
Suppose $\mu$ satisfies~\eqref{eq:theta_choice_mu_improved}.
For every $n\geq 2$ it holds that
\begin{equation}
\label{eq:temporal_half_cross}
\tilde{t}_{n}
    \le \frac{C 2^{dn}}{f(b_{n-1})} +
        (3^{d} t_{n-1} + 2d \cdot 3^{d-1} h_{n-1})^{2}.
\end{equation}
\end{lema}

\begin{proof}
For $v \in \{0, 1, 2, 3\}$ let us define
\begin{equation*}
    I_v := 2^{n-2}v + [0,2^{n-2}].
\end{equation*}
This collection of $4$ intervals of length $2^{n-2}$ covers
$[0,2^{n}]$. On $T([0,2^{n}]^{d} \times [0,b_{n-1}])$ we can choose a
path $\gamma: [0,b_{n-1}] \to [0,2^{n}]^{d}$ that realizes the temporal
crossing and consider its range $\cI = \gamma([0,b_{n-1}])$. Project set
$\cI$ in each coordinate direction $j$, obtaining a discrete interval
$\cI_j \subset [0,2^{n}]$, and define the \textit{box count} of $\cI_j$
by
\begin{equation}
\label{eq:projection_box_count}
c_j := \min\{|I|;\;
    I \subset \{0,1,2,3\},\ \cI_j \subset \cup_{v\in I} I_v\}.
\end{equation}
We decompose our event with respect to what is observed on each $\cI_j$.

If for every $1\le j \le d$ we have $c_j \le 2$ then the whole path
$\gamma$ is contained inside a $d$-dimensional box with side length
$2^{n-1}$. In this case, we have some choice of $v \in \{0,1,2\}^{d}$
such that
\begin{equation*}
    \cI \subset 2^{n-2} v + [0,2^{n-1}]^{d},
\end{equation*}
and the number of possible $v$ is given by $3^{d}$.

Now, let us consider the case in which some $c_j \geq 3$ and thus
$\cI$ is not contained in some of the boxes with side length $2^{n-1}$
described above. In this case, we refine the argument by considering time. For
any time $t \in [0,b_{n-1}]$ we define $\cI(t) := \gamma([0,t])$ and for any
fixed direction $j$ we consider its projection $\cI_j(t)$ and its
box count $c_j(t)$. Define
\begin{equation*}
t_1 := \inf\{t \in [0,b_{n-1}];\; \exists 1 \le j \le d
        \text{ such that $c_j(t) \geq 3$}\}.
\end{equation*}
Since $\gamma$ can only change value when there is transmission to a neighbouring
site, at time $t_1$ we have $c_{j_0}(t_1-) = 2$ and $c_{j_0}(t_1) = 3$ for some
special direction $j_0$ and $c_j(t_1) \le 2$ for every other direction. Thus, there
is $v \in \{0,1,2\}^{d}$ such that
\begin{align*}
    \cI(t_1-) \subset 2^{n-2} v + [0,2^{n-1}]^{d}, \quad
    &\text{but} \quad
    \cI_{j_0}(t_1) \nsubseteq 2^{n-2} v + [0,2^{n-1}]^{d} \\
    &\text{and} \quad
    c_{j_0}(t_1) = 3.
\end{align*}
Notice that this means path $\gamma$ must have crossed a half-box of
$2^{n-2} v + [0,2^{n-1}]^{d}$ on direction $j_0$ during time
interval $[0,t_1] \subset [0, b_{n-1}]$, see Figure~\ref{fig:temp_project_gen}.
There are $2d \cdot 3^{d-1}$ possible half-boxes to be crossed, which implies
\begin{equation*}
\hat{\PP}(T([0,2^{n}]^{d} \times [0, b_{n-1}]))
    \le 3^{d} t_{n-1} + 2d \cdot 3^{d-1} h_{n-1}.
\end{equation*}
Since the bound above holds for any choice of renewal starting points
$\{\mathfrak{t}_x; \mathfrak{t}_x \le 0, x \in [0,2^{n}]^{d}\}$, taking the
supremum over all such collections the result follows.
\end{proof}

\begin{figure}
\centering
\begin{tikzpicture}[scale=1,
    dot/.style={fill, minimum size=3pt, outer sep=0pt,
                    inner sep=0pt, circle, blue},
    z={(55:.7cm)}
    ]
    \draw[dashed, help lines] (0,0) grid (4,4);
    \draw[very thick] (0,0) rectangle (4,4);
    \draw[|<->|] (0,-.2) -- (4,-.2) node[midway, below] {$4 \cdot 2^{n-2}$};
    \draw[very thick, red] (2,1) rectangle (3,3);
    \fill[opacity=.5, red] (2,1) rectangle (3,3);

    \coordinate[dot] (A) at (2.3, 1.2);
    \coordinate (B) at (2.8, 1.4);
    \coordinate[dot] (t1) at (3.1, 2.7);
    \node[above right] at (t1) {$\gamma(t_1)$};
\draw[very thick, tension=.7, blue, ->, >=stealth] plot [smooth] coordinates
    {(A) (2.5, 1.9) (2.2, 2.4) (2, 1.5) (1.6, 2) (2, 2.4)
    (t1) (3.3, 2.5) (2.8,2.2) (B)};

% 3d half-box
\draw[very thick, red] (7,0)
    -- ++( 1,0,0) -- ++(0,0, 2) -- ++(0,3,0)
    -- ++(-1,0,0) -- ++(0,0,-2) -- cycle; %half-box frame
\draw[red, dashed] (7,3,2) -- (7,0,2) -- (8,0,2) (7,0,2) -- (7,0,0);
\draw[very thick, red] (8,0,0) -- (8,3,0) -- (7,3,0) (8,3,0) -- (8,3,2);
 %time length
\draw[|<->|] (6.8,0) -- (6.8,3) node[midway, left] {$b_{n-1}$};
%spatial lengths
\draw[|<->|] (7,-.2) -- (8,-.2) node[midway, below] {$2^{n-2}$};
\draw[|<->|] (8.2,0,0) -- (8.2,0,2) node[midway, below right] {$2^{n-1}$};

% 3d path
\coordinate[dot] (3dA)  at (7.3, 0.0, 0.2);
\coordinate      (3dB)  at (7.8, 3.0, 0.4);
\coordinate[dot] (3dt1) at (8.1, 2.0, 1.7);

\node[below right] at (3dt1) {$(\gamma(t_1), t_1)$};
\draw[very thick, tension=.7, blue, ->, >=stealth] plot [smooth] coordinates {
    (3dA) (7.5, 0.3, 0.9) (7.2, 0.7, 1.4) (7.0, 1.0, 0.5) (6.6, 1.3, 1.0)
    (7.0, 1.7, 1.4) (8.1, 2.0, 1.7) (8.3, 2.3, 1.5) (7.8, 2.7, 1.2) (3dB)};
\end{tikzpicture}
\caption{Depiction of the argument in Lemma~\ref{lema:temporal_half_cross}
for the case $d=2$. When the space projected temporal crossing is not
contained in one of the $3^{d}$ sub-boxes of side length $2^{n-1}$ we
must have a spatial crossing of a half-box of scale $n-1$.}
\label{fig:temp_project_gen}
\end{figure}
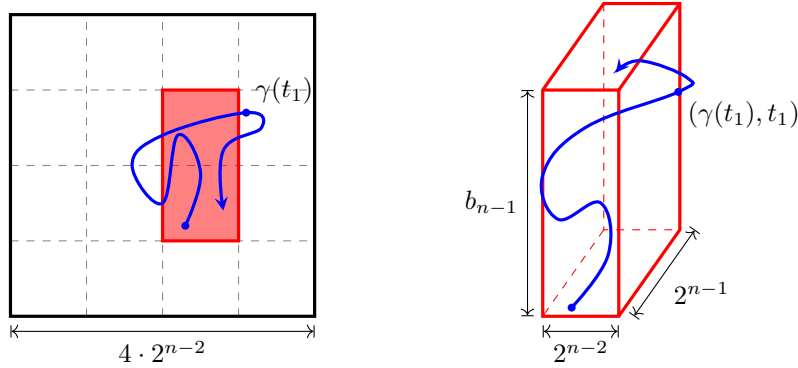

\medskip
\noindent
\textbf{Spatial crossing.} Now we prove a similar bound for quantity $h_n$.
Recall that independence of the Poisson processes implies that for crossing
$B_n$ in some fixed spatial direction we need to perform two independent
crossings of half $B_n$ in that direction, implying
\begin{equation*}
s_n \le h_n^{2}.
\end{equation*}
A similar bound for $h_n$ implies the following lemma.

\begin{lema}[Spatial Crossing]
\label{lema:spatial_cross}
For $n\geq 2$ it holds that
\begin{equation}
\label{eq:spatial_cross}
h_n
    \le 4 \cdot 36^{d-1} \cdot
        \Bigl\lceil \frac{b_n}{b_{n-1}} \Bigr\rceil^{2} \cdot
        (h_{n-1} + \tilde{t}_{n-1})^2.
\end{equation}
\end{lema}

\begin{proof}
Independence of Poisson processes implies that
\begin{equation*}
h_n
    \le \sup\hP\bigl(
            S_1([0,2^{n-2}] \times [0,2^{n}]^{d-1} \times [0,b_n])
        \bigr)^2.
\end{equation*}

Let us simplify notation here. Since in a first moment we will work with boxes
with time length $[0,b_n]$ we omit it from the notation. Also, on space
coordinates we only work with intervals of length $2^{n}, 2^{n-1}$ or
$2^{n-2}$, so we write simply
\begin{equation*}
B(l_1, \ldots, l_d)
    = \Bigl(\smash{\prod_{i=1}^{d}} [0,2^{n-l_i}] \Bigr) \times [0,b_n]
    \quad \text{for $l_i \in \{0,1,2\}$}.
\end{equation*}
We refer to a crossing of such box on direction $j$ as
$S_j(l_1, \ldots, l_d)$. Using this notation we want to show that
on $S_1(2,0,\ldots,0)$ we can find some crossing of boxes whose side
lengths are all at most $2^{n-1}$, leading to an estimate of the form
\begin{equation*}
\hat{\PP}\bigl(S_1(2,0,\ldots,0)\bigr)
    \le C(d) \cdot \hat{\PP}\bigl( S_1(2,1,\ldots,1) \bigr),
\end{equation*}
recalling that $\hat{\PP}$ refers to a probability measure starting from some
fixed collection $\{\mathfrak{t}_x; \mathfrak{t}_x \le 0, x \in \ZZ^{d}\}$ of starting renewal
marks. The main step in this simplification is the following. Consider event
$S_1(2, l_2, \ldots, l_d)$ and suppose that in direction $j$ we have
$l_j = 0$, meaning that the interval length in that direction is $2^{n}$.
Consider a path $\gamma: [s_1,t_1] \to \ZZ^{d}$ with
$[s_1,t_1] \subset [0,b_n]$ that realizes event
$S_1(2, l_2, \ldots, l_d)$ and let $\cI_j$ be the projection of
$\gamma([s_1,t_1])$ on direction $j$ and $c_j$ be its box count, i.e.,
\begin{equation*}
c_j := \min\{|I|;\;
    I \subset \{0,1,2,3\},\ \cI_j \subset \cup_{v\in I} I_v\}.
\end{equation*}

When $c_j \le 2$ we can ensure that $\cI_j$ is contained in
$[v2^{n-2}, (v+2)2^{n-2}]$ for some $v \in \{0,1,2\}$.
Thus, instead of the original box $B(2,l_2, \ldots, l_d)$
we can observe the same crossing on the smaller box in which on direction
$j$ we replace $[0,2^{n}]$ by $[v2^{n-2}, (v+2)2^{n-2}]$, an interval with
length $2^{n-1}$.
Similarly, if $c_j \geq 3$ we know that $\cI_j$ must have crossed either
$I_{1}$ or $I_{2}$, implying the crossing on direction $j$ of a smaller
box, since now the interval length on direction $j$ is $2^{n-2}$.

\begin{figure}
\centering
\begin{tikzpicture}[scale=1,
    dot/.style={fill, minimum size=3pt, outer sep=0pt,
                    inner sep=0pt, circle, blue},
    z={(55:.7cm)}
    ]

% coordinate axis for 3d
\coordinate (o1) at (-2.6,0,2);
\draw[->] (o1) -- ++(.7,0) node[below] {$\overrightarrow{e_1}$};
\draw[->] (o1) -- ++(0,.7) node[left] {$t$};
\draw[->] (o1) -- ++(0,0,-.6) node[below] {$\ZZ^{d-1}$};

% 3d half-box
\draw[very thick] (0,0)
    -- ++( 2,0,0) -- ++(0,0, 2) -- ++(0,3,0)
    -- ++(-2,0,0) -- ++(0,0,-2) -- cycle; %half-box frame
\draw[dashed] (0,3,2) -- (0,0,2) -- (2,0,2)
    (1,0,0) -- (1,0,2) -- (1,3,2) (0,0,2) -- (0,0,0);
\draw[very thick] (2,0,0) -- (2,3,0) -- (0,3,0)
    (1,0,0) -- (1,3,0) -- (1,3,2)
    (2,3,0) -- (2,3,2);
 %time length
\draw[|<->|] (-.2,0) -- (-.2,3) node[midway, left] {$b_{n}$};
%spatial lengths
\draw[|<->|] (0,-.2) -- (2,-.2) node[midway, below] {$2 \cdot 2^{n-2}$};
\draw[|<->|] (2.2,-.1,0) -- (2.2,-.1,2)
    node[midway, below right] {$[0,2^{n}]^{d-1}$};

% 3d path
\coordinate[dot] (A) at (0, 1  ,0.4);
\coordinate      (B) at (2, 1.5,1.4);
\draw[very thick, tension=.7, blue, ->, >=stealth] plot [smooth] coordinates
    {(A) (.5,1.3) (1,1.1) (1.9, 1.6) (B)};

% coordinate axis for projection
\coordinate (o) at (4.7,0);
\draw[->] (o) -- ++(.7,0) node[below] {$\overrightarrow{e_j}$};
\draw[->] (o) -- ++(0,.7) node[left] {$t$};

% projection on direction j
\begin{scope}[shift={(6,0)}]
\draw[very thick] (0,0) rectangle (2.8,4);
\draw[dashed] (0,0) grid[xstep=.7,ystep=4] (2.8,4);
% lengths
\draw[|<->|] (0  ,-.2) -- (1.4,-.2) node[midway, below] {$2^{n-1}$};
\draw[|<->|] (1.4,-.2) -- (2.8,-.2) node[midway, below] {$2^{n-1}$};
\draw[|<->|] (0.7,4.2) -- (2.1,4.2) node[midway, above] {$2^{n-1}$};
\draw[|<->|] (3,0) -- (3,4) node[midway, right] {$b_{n}$};

\draw[very thick, tension=.7, blue, ->, >=stealth] plot [smooth] coordinates
    {(.2, .9) (.8, 1) (1.2, 1.5) (1.6, 2)};
\draw[very thick, tension=.7, blue, ->, >=stealth] plot [smooth] coordinates
    {(2, .9)  (2.5, 1.5) (1.8, 2.5) (2.2, 3)(1.6, 3.5)};
\end{scope}
\end{tikzpicture}
\caption{Crossing of a half box at scale $n$ implies two independent
spatial crossings. For each crossing, on direction $2 \le j \le d$ there are 2
possibilities: either the crossing traverses some interval of length $2^{n-2}$
or it remains inside an interval of length $2^{n-1}$.}
\label{fig:spatial_project}
\end{figure}
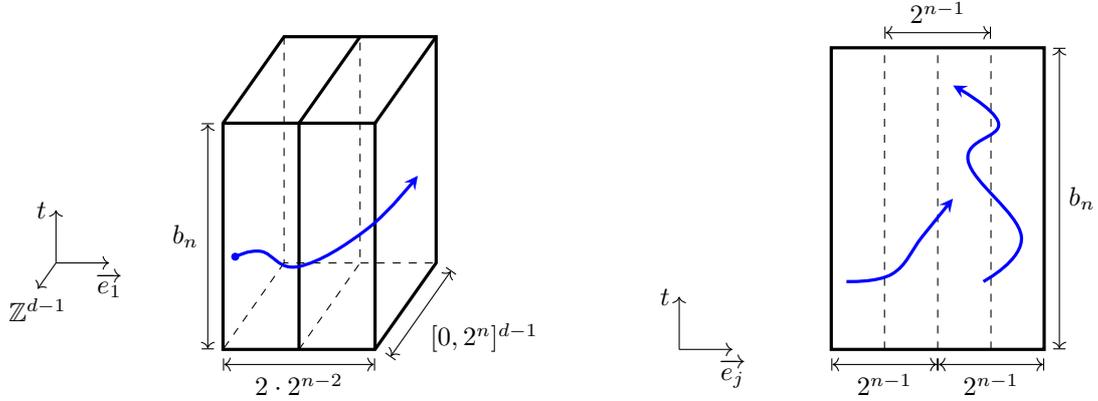
In both cases, the crossing of our original box implies the occurrence
of some crossing of a smaller box inside it, see
Figure~\ref{fig:spatial_project}. Abusing notation, we do not
specify the exact position of these smaller boxes, since in the final
bound we use the uniform quantities from~\eqref{eq:uniform_quantities}.
Thus, we have
\begin{align*}
    \smash{\hat{\PP}}(S_1&(2, l_1, \ldots, 0, \ldots, l_d)) \\
    &\le 3 \smash{\hat{\PP}}(S_1(2, l_1, \ldots, 1, \ldots, l_d)) +
        2 \smash{\hat{\PP}}(S_j(2, l_1, \ldots, 2, \ldots, l_d)) \\
    &=   3 \smash{\hat{\PP}}(S_1(2, l_1, \ldots, 1, \ldots, l_d)) +
        2 \smash{\hat{\PP}}(S_1(2, l_1, \ldots, 2, \ldots, l_d))
\end{align*}
where the equality above follows from symmetry.
For $l \in \{1,2\}^{d-1}$ let us denote $a(l) = \#\{i;\; l_i = 1\}$.
Applying the reasoning above to directions $2 \le j \le d$ successively,
we can write
\begin{equation*}
\smash{\hat{\PP}}(S_1(2, 0, \ldots, 0))
    \le \sum_{l \in \{1,2\}^{d-1}} \smash{\hat{\PP}}(S_1(2, l)) \cdot
        3^{a(l)} \cdot 2^{d-1-a(l)}.
\end{equation*}
Finally, notice that any $\smash{\hat{\PP}}(S_1(2, l))$ with
$l \in \{1,2\}^{d-1}$ is upper bounded by
$\smash{\hat{\PP}}(S_1(2,1,\ldots,1))$ since increasing the box in
some direction $2\le j \le d$ can only make it easier to find a crossing.
This leads to the bound
\begin{align*}
\smash{\hat{\PP}}(S_1(2, 0, \ldots, 0))
    &\le \smash{\hat{\PP}}(S_1(2,1,\ldots,1)) \cdot 2^{d-1}
        \sum_{\mathclap{l \in \{1,2\}^{d-1}}} \, (3/2)^{a(l)}\\
    &\le 6^{d-1} \cdot \smash{\hat{\PP}}(S_1(2,1,\ldots,1)).
\end{align*}
Returning to our previous notation, now we want to bound
\begin{equation*}
\smash{\hat{\PP}}(S_1(2,1,\ldots,1))
    = \smash{\hat{\PP}}(
        S_1([0,2^{n-2}] \times [0,2^{n-1}]^{d-1} \times [0,b_n]))
\end{equation*}
in terms of $h_{n-1}$ and so we need to fix the time scale above.
We use a collection of overlapping boxes
\begin{equation*}
R_i
    = [0,2^{n-2}] \times [0,2^{n-1}]^{d-1} \times
        [i b_{n-1}, (i+1)b_{n-1}]
    \quad \text{for $i \in \tfrac{1}{2}\ZZ$}
\end{equation*}
to cover the time interval $[0, b_n]$.
Then, either our path $\gamma$ ensures we have $S_{1}(R_i)$ for some $i$ or it
must make a temporal crossing of some box
$[0,2^{n-2}] \times [0,2^{n-1}]^{d-1} \times [i b_{n-1}, (i+1/2)b_{n-1}]$,
which is event $\tT(R_i)$. Thus, we can write
\begin{equation*}
\smash{\hat{\PP}}(S_1([0,2^{n-2}] \times [0,2^{n-1}]^{d-1} \times [0,b_n]))
    \le 2 \Bigl\lceil \frac{b_n}{b_{n-1}} \Bigr\rceil \cdot
        (h_{n-1} + \tilde{t}_{n-1}).
\end{equation*}
Putting the bounds above together and taking the supremum over all possible
collections of starting times, we obtain~\eqref{eq:spatial_cross}.
\end{proof}

\medskip
\noindent
\textbf{Simplifying recurrence.} Looking at the expressions obtained in
Lemmas~\ref{lema:temporal_half_cross} and~\ref{lema:spatial_cross},
it seems useful to work with a simpler recurrence based on the quantity
\begin{equation*}
    u_n := h_n + \tilde{t}_n.
\end{equation*}
Noticing that $t_n \le \tilde{t}_n$ we can write
\begin{align}
u_{n}
    &\le \bigl[
            C(d) \cdot (b_n/b_{n-1})^{2} \cdot (h_{n-1} + \tilde{t}_{n-1})^2
        \bigr] +
        \Bigl[
            C(d)(\tilde{t}_{n-1} + h_{n-1})^{2} + \frac{C 2^{dn}}{f(b_{n-1})}
        \Bigr] \nonumber\\
\label{eq:recurrence_un}
    &\le C(d) \cdot (b_n/b_{n-1})^{2} \cdot u_{n-1}^{2} +
        \frac{C 2^{dn}}{f(b_{n-1})}.
\end{align}

\begin{lema}
\label{lema:un_estimate_specific}
Let $\mu$ be any probability distribution on $\RR_+$ and $\cR$ be a renewal
process with interarrival $\mu$ started from some $\mathfrak{t} \le 0$.
Suppose
\begin{equation}
\label{eq:f_moment}
    \int_{1}^{\infty} x e^{\theta (\ln x)^{1/2}} \mu(\mathrm{d}x)
    \quad \text{for some $\theta > 4 \sqrt{d \ln 2}$}.
\end{equation}
There is  a choice of sequence $b_n$ and a natural number
$n_0(\mu, \theta, d)$ such that if
$u_{n_0} \le 2^{-d n_0}$ then for every $n \geq n_0$ we have
$u_n \le 2^{-d n}$. Consequently, there exists
$\lambda_0(\mu, \theta, d) > 0$ such that $\PP(\tau^{0} = \infty) = 0$ for any
$\lambda \in (0, \lambda_0)$.
\end{lema}

\begin{proof}
Consider the sequence of boxes $B_n = [0,2^{n}]^{d} \times [0,b_n]$.
Recall function $f$ is given by
\begin{equation*}
f(x) := e^{\theta (\ln x)^{1/2}} \cdot \I\{x \geq 1\}.
\end{equation*}
We want to take $f(b_{n-1}) := e^{\alpha (n-1)}$ for $\alpha > 0$ a
parameter to be chosen later so that $2^{nd}/f(b_{n-1})$ tends to zero
sufficiently fast. This can be accomplished by taking
$b_n := e^{(\alpha/\theta)^{2} n^{2}}$.
Recurrence relation~\eqref{eq:recurrence_un} then becomes
\begin{equation*}
u_{n}
    \le C(d) {\Bigl( \frac{b_n}{b_{n-1}} \cdot u_{n-1}\Bigr)}^{2} +
        C(\mu, \theta) \exp[ (d \ln 2) n - \alpha (n-1) ].
\end{equation*}
Because of the error term above, the decay of $u_n$ cannot be faster than
$e^{-\alpha (n-1)}$. Based on this, we suppose
$u_{n-1} \le e^{-\beta (n-1)}$ for some parameter $\alpha > \beta > 0$.
Under this assumption we can estimate
\begin{equation*}
\Bigl(\frac{b_n}{b_{n-1}} \cdot u_{n-1}\Bigr)^{2}
    = \Bigl(e^{(\alpha/\theta)^2 (2n-1)} \cdot u_{n-1}\Bigr)^2
    \le e^{ 2(\alpha/\theta)^2 (2n-1)-2 \beta (n-1)},
\end{equation*}
which leads to
\begin{align}
u_{n}
    &\le C(d) e^{2(\alpha/\theta)^{2} (2n-1) - 2 \beta (n-1)} +
        C e^{(d \ln 2 - \alpha)n + \alpha} \nonumber\\
\label{eq:un_recurrence_simpler}
    &\le C(d, \alpha, \beta, \theta) e^{[4(\alpha/\theta)^{2} - \beta] n}
        \cdot e^{-\beta n} +
        C(\mu, \theta, \alpha) e^{(\beta + d \ln 2 - \alpha)n} \cdot e^{-\beta n}.
\end{align}
The induction will follow once we ensure
\begin{equation*}
\begin{cases}
4 (\alpha / \theta)^{2} - \beta < 0 \\
\beta + d \ln 2 - \alpha < 0
\end{cases}
\quad \text{or, equivalently,} \quad
\begin{cases}
\theta^{2} > \frac{4 \alpha^{2}}{\beta} \\
\alpha > \beta + d \ln 2
\end{cases}.
\end{equation*}
We want to choose parameters $\alpha, \beta$ in order to make $\theta$ as
small as possible while still being able to perform the induction.
Notice that combining the two inequalities above we have
\begin{equation*}
\theta^2
    > 4 \Bigl(\sqrt{\beta} + \frac{d \ln 2}{\sqrt{\beta}}\Bigr)^2
    \geq 16 d \ln 2,
\end{equation*}
by AM-GM inequality, with equality when $\beta = d \ln 2$. So,
hypothesis~\eqref{eq:f_moment} is the best we can hope in this setup.
Fix $\beta = d \ln 2$. Looking at the possible values of $\alpha$,
we need to choose
\begin{equation*}
2d \ln 2 < \alpha < \sqrt{\frac{\theta^{2} d \ln 2}{4}}.
\end{equation*}
Since~\eqref{eq:f_moment} implies
$\sqrt{\frac{\theta^{2} d \ln 2}{4}} > 2d \ln 2$, we can take for instance
$
\alpha(d, \theta)
:= \frac{1}{2} \Bigl(2d \ln 2 + \sqrt{\frac{\theta^{2} d \ln 2}{4}}\Bigr)
$.
Take $n_0 = n_0(\mu, d, \theta)$ sufficiently large so that
\begin{equation}
\label{eq:n0_conditions}
\begin{cases}
C(d, \alpha, \beta, \theta) e^{[4(\alpha/\theta)^{2} - \beta] n}
    &\le \frac{1}{4} \\
C(\mu, \theta, \alpha) e^{(\beta + d \ln 2 - \alpha)n}
    &\le \frac{1}{4}
\end{cases},
\quad \text{for all $n \geq n_0$}.
\end{equation}
This is possible since both left hand sides tend to zero as $n \to \infty$.
Suppose that $u_{n_0} \le e^{-\beta n_0} = 2^{-d n_{0}}$, recalling that
$\beta = d \ln 2$. Then, we have by~\eqref{eq:un_recurrence_simpler} that
\begin{equation*}
u_{n}
    \le \frac{1}{4} e^{-\beta n} + \frac{1}{4} e^{-\beta n}
    \le e^{-\beta n}
    \quad \text{for every $n \geq n_0$}.
\end{equation*}

The induction just described will hold if we can ensure that the base case
$n = n_0$ holds. But if $n_0(\mu, \theta, d)$ is fixed we can take
$\lambda_0$ sufficiently small for it to hold. Indeed, just notice that for any
box $(x, t) + B_{n_0}$ if we denote by $N$ the number edges contained in
$[0,2^{n_0}]^{d}$ we have that
\begin{equation*}
\hP(H)
    := \hP(\text{no transmission on $(x,t) + B_{n_0}$})
    = e^{- \lambda b_{n_0} \cdot N} \to 1
\end{equation*}
as $\lambda \to 0$.
Clearly, we have  $h_{n_0} \le \hP(H^\comp)$.
Moreover, we can control $\tilde{t}_{n_0}$ similarly, since 
if there is no transmission the only
possibility for a temporal half-crossing of box $(x,t) + B_{n_0}$
is achieving it by a single
site, an event which we recall was denoted $J = J_{n_0}(t,
b_{n_{0}}/2)$ in
Corollary~\ref{coro:moment_condition}. Hence, we have
\begin{align*}
\hP(H \cap J)
    &\le \frac{C(\mu, \theta) 2^{d n_{0}}}{f(b_{n_0}/2)}
    = C2^{d n_{0}}
        \exp\Bigl[- \theta \Bigl(\frac{\alpha^2}{\theta^2} n_0^2 - \ln
        2\Bigr)^{\frac{1}{2}}\Bigr] \\
    &= C 2^{d n_{0}}
        \exp\Bigl[ - \alpha n_0 \Bigl(1 - \frac{\theta^2 \ln
        2}{(\alpha n_0)^2}\Bigr)^{\frac{1}{2}}\Bigr]
    =  C e^{(d \ln 2 - \alpha) n_0 + O(n_0^{-1})}
\end{align*}
as $n_0 \to \infty$. Making a small change in~\eqref{eq:n0_conditions},
we can increase $n_0$ if needed to ensure $\hP(H \cap J) \le \frac{1}{4} e^{- \beta n_0}$
and write
\begin{equation*}
\max \{h_{n_0}, \tilde{t}_{n_0}\}
    \le \sup \{\hP(H^{\comp}) + \hP(H \cap J)\}
    \le 1 - e^{- \lambda b_{n_0}N} + \frac{1}{4} e^{- \beta n_0}
    \le \frac{1}{2} e^{- \beta n_0}
\end{equation*}
for $\lambda$ sufficiently small. We conclude
$u_{n_0} \le e^{- \beta n_0}$.
\end{proof}

\begin{proof}[Proof of Theorem~\ref{teo:phase_transition_improved}]
It follows from the conclusion of Lemma~\ref{lema:un_estimate_specific}.
\end{proof}

\begin{remark}
\label{condition}
The exponent $1/2$ in the definition of function $f$ is the best possible,
meaning that the same reasoning does not work for a function
$g = \exp[\theta (\ln x)^{\delta}]$ with $\delta < 1/2$.
\end{remark}

\begin{remark}
\label{finite}

We recall (see e.g. Section 2 of \cite{FMV}) that when the interarrival
distribution $\mu$ has a density $f$ and $h_\mu(t)=f(t)/ \mu(t,\infty)$ is
the hazard rate function, the corresponding renewal process starting at
some point $t_0 \in \mathbb{R}$ can be easily obtained in terms of a
homogeneous Poisson point process on $\mathbb{R} \times \mathbb{R}_+$ with
intensity 1. The construction shows that when the hazard rate is
decreasing, the corresponding renewal point process, hereby denoted by
$\mathcal{R}_\mu$, is an increasing function of points in the Poisson point
process. As already mentioned, this property was used in \cite{FMV} to
guarantee the FKG property. Moreover, as easily verified, it also yields
the following:

If $\nu$ is another probability measure on $(0, \infty)$ with a density
$g$ and hazard rate $h_\nu$, and ${h_\nu(t) \ge h_\mu(t)}$ for all $t$, then
the two renewal processes starting at some $t_0 \in \mathbb{R}$ can be
coupled in such a way that $\mathcal{R}_\mu \subset \mathcal{R}_\nu$ with
probability one.

Using this observation with $\nu$ being an exponential distribution, we
conclude that if $h_\mu$ is decreasing and bounded, then $\mathcal{R}_\mu$
can be embedded in a Poisson point process. Thus, the classical result on
Harris contact process yields $\lambda_c(\mu)< \infty$. It is easy to come
up with a wide range of  examples of such $\mu$'s.
\end{remark}

\section{Complete convergence}
\label{sec:complete_convergence}
% previously \label{sec; 4}

In this section we prove Theorem \ref{teo:complete_convergence}, which relies
on a variant of the argument of \cite{FMMV}. We begin with a sketch of the
argument. In the following, $\lambda$ is a fixed strictly positive infection rate.
%Our argument is to show that in the event that the process survives there must be
%times in which a site percolates in the manner that is shown in proof of
%Theorem 1 in \cite{FMMV}.
For our RCP equipped with its natural filtration
$(\cA_t)_{t\ge 0}$, we say a stopping time $T$ is \emph{extreme} if
\begin{equation*}
\max \bigl\{ \Vert x \Vert_\infty:\; \xi_T (x) = 1 \bigr\}
    > \smash{\max_{s<T}} \bigl\{ \Vert x\Vert_\infty:\; \xi_s (x) = 1 \bigr\},
\end{equation*}
where $\Vert x \Vert_{p}$ denotes the usual $\ell^p$-norm on $\ZZ^{d}$.

An extreme stopping time $T$ is useful as it implies the existence of a site
$x_T$ such that $\xi_T(x_T) = 1$ and a Euclidean unit vector $\vec e$ in
$\ZZ^d$ so that all renewal processes
$(\cR_{x_T + m \vec e}\, ;\; m \geq 1)$
are conditionally i.i.d.\ independent of $\cA_T$.
%\noindent
%or
%\noindent all renewal processes on $( - \infty, x_T]$ are conditionally
%independent of $\cA_T$.

On Lemma~\ref{lema:infinite_extremes} we show that the probability
of the infection surviving only inside a finite cylinder is zero. This means
that we can find arbitrarily large extreme stopping times, which gives us a way
to build independent events.

Then, on Corollary~\ref{coro:lasting_infection} we build a tunneling event inspired
in the construction of~\cite{FMMV}. We prove that almost surely we can find a
sequence of (random) sites that do not have any cure marks for a very long time.
These intervals have a sizeable overlap, allowing the infection to travel
from one to another and making them to work as hubs for sustaining the
infection indefinitely. Corollary~\ref{coro:lasting_infection} provides a
quantitative control for the space-time position of such hubs.

Finally, on Lemma~\ref{lema:freely_infects} we leverage the existence of
these hubs and show they can be used to infect every site in a suitably
sized region around the hub. This construction is based on the following
definition.
\begin{defi}
\label{defi:freely_infects}
We say $(x, u)$ freely-infects $(y, v)$ in the set $A  \subset \ZZ^d$ if
there exists a sequence of points $x = x_0, x_1, \dots, x_n = y$ and times
$u < t_1 < t_2 < \ldots <t_n < v$ so that for each $i$ the sites
$x_{i - 1}$ and $x_i$ are nearest neighbours and $x_i \in A$,
and there is an infection mark from $x_{i - 1}$ to $x_i$ at time $t_i$.
\end{defi}

We stress that we are not assuming that $\xi_u (x) = 1$. The event
``$(x, u)$ freely-infects $(y, v)$ in $A$'' depends purely on the collection
of Poisson processes and does not concern the renewal processes, i.e. it
does not take into account the recovery times.

The proof of Theorem~\ref{teo:complete_convergence} puts all these pieces
together to conclude that, on the event that the infection survives, the
probability that any fixed finite set $K \subset \ZZ^d$ is infected at a large
time tends to 1.

\begin{lema}
\label{lema:infinite_extremes}
% previously \label{lem4.1}
For a RCP with $\tau = \inf\{s > 0 : \xi_s \equiv 0 \}$ we have
\begin{equation*}
P \Bigl(
    \{\tau = \infty \} \cap
    \Bigl\{
        \bigl| \{x : {\textstyle\int}_{0}^{\infty} \xi_s (x) ds > 0\} \bigr| < \infty
    \Bigr\}
    \Bigr) = 0.
\end{equation*}
\end{lema}
This lemma implies that for all time $t$ there a.s. exists on the set
$\{ \tau = \infty \}$ an extreme time $T > t$. Indeed, just take the next time
after $t$ when the process encounters a site whose norm is a new maximum
among sites infected or previously infected.

\begin{proof}
Without loss of generality we assume that $\sum_x \xi_0 (x) < \infty$ as
otherwise the result is trivial.
The idea is that if the result were not true, then for some $m< \infty $ the infected sites would remain a subset of $[-m,m]^d $ for ever.  But as time becomes large there will be arbitrarily large intervals of time on which there are no renewal points for any $x \in [-m,m]^d $ and so no opposition to the process infecting sites outside this
finite set.
It is enough to prove that for each $m \in \NN$, the event
\begin{equation*}
\{ \tau= \infty \} \cap
    \bigl\{\xi_s(x) = 0,\ \forall s \ge 0,\ \forall x \notin [-m,m]^d \bigr\}
\end{equation*}
has probability zero. But by Proposition 7 in \cite{FMMV}, for all $n$ large
enough the probability that at least one of the renewal processes on
$[-m, m]^d$ intersects the time interval $[2^n, 2^n + 2^{n \epsilon_1}]$ is
less than $(2m + 1)^d 2^{-n \epsilon_1}$,  provided $\epsilon_1$ is fixed
strictly positive but small enough. Furthermore, the probability
(conditional upon $\xi_{2^n} (x) = 1$ for some $x \in [-m, m]^d$)
that there does not exist a sequence $x = x_0, x_1 \ \cdots \ x_{k}$
with $k \leq m+1$ of nearest neighbour sites that satisfy

\begin{enumerate}[(i)]
\item $\xi_{2^n + 2^{n \epsilon_1}}(x_{k} ) = 1$, with $  x_{k}  \notin [-m, m]^d$;

\item $\lVert x_i - x_{i-1} \rVert_1 = 1$, for all $1 \leq i \leq k$; and

\item $\xi_{2^n}(x_0) = 1$ with $x_0 \in [-m, m]^d$.
%$\xi_{2^n + 2^{n \epsilon_1}}(x_{k} ) \ = \ 1, \ \quad $ with $  x_{k}  \notin [-m, m]^d$
\end{enumerate}
tends to zero as $n \rightarrow \infty$, which implies the lemma.
\end{proof}

\begin{coro}
\label{coro:lasting_infection}
%previously \label{cor5.1}
On the event $\{\tau = \infty \}$, for all $t$ large there is a site
$x_t$ within distance $\ln^3 t$ of the origin so that
$\xi_s(x_t) = 1$ for all $s \in [t/2, t]$.
\end{coro}

\begin{proof}
For purely notational reasons we suppose that the dimension, $d$, is equal to
one. We first treat the case $\sum_x \xi_0(x) < \infty$ and then note how the
argument given can be extended to the infinite case.
Given an extreme stopping time $T$, we define a suitable tunnelling
event $H_T$ (see Figure~\ref{fig:lasting_infection}). What is important
is that its conditional probability given $\cA_T$ should be bounded away
from zero on $\{T > N\}$, where $N$ is a large constant. In this description
and calculation of probability bounds, we suppose
\begin{equation*}
    X_T > \max \{ x : \text{$\exists  s<T$ so that $\xi_ s (x) = 1$}\}.
\end{equation*}
If $X_T< \min \{  x : \text{$\exists  s<T$ so that $\xi_ s (x) = 1$}\}$
then we simply reflect the definitions and all probability bounds
will be the same. Define
\begin{equation}
\label{eq5.1}
    H_T := \bigcap \limits_{n = 0}^\infty H_{T, n},
\end{equation}
where the events $H_{T,n}$ are defined recursively via the random integers
$\{L_j \}_{j=0}^{\infty}$ and $\{n_j\}_{j=0}^{\infty}$:
$H_{T,0}$ is simply the event
$\{\cR_{X_T} \cap [T, 2^{n_0+2}] = \emptyset\}$, where
$2^{n_0} = \inf\{ 2^n:\; 2^n >T\}$. By Lemma 2 in \cite{FMMV},
there exists $c>0$ so that
\begin{equation*}
    P(H_{T,0}| \cA_T) \geq c >0 \quad \text{on $\{T > N\}$}
\end{equation*}
for all $N$ fixed.
We take $L_0 := 0$. Given $n_0,\dots, n_{i - 1}$ and
$L_0, L_1, \dots,  L_{i - 1}$ we set $n_i = n_{i - 1} + 1$ and define
\begin{equation*}
L_i := \inf \{ k > L_{i - 1} :
    \cR_{X_T+ k} \cap [2^{n_{i - 1}}, 8 \cdot 2^{n_{i - 1}}] = \emptyset \}.
\end{equation*}

\noindent
Our event $H_{T,i}$ is given by the following conditions:
\begin{enumerate}[(i)]
\item $L_i - L_{i - 1} \leq i n_0$;

\item There exists an infection path from
    $(X_T + L_{i - 1}, 2^{n_{i - 1}} )$ to $(X_T+ L_i, 2^{n_i})$
    in the space-time rectangle
    $[X_T + L_{i - 1},  X_T+ L_i] \times  [2^{n_{i - 1}}, 2^{n_i}]$.
\end{enumerate}

From the argument in Section 4 of \cite{FMMV}, we have that if
$N$ is fixed sufficiently large then
\begin{equation*}
P (H_T\mid \cA_T)
    = P\Bigl(\; \smash{\bigcap_{i=0}^{\infty}} H_{T,i} \Bigm| \cA_T \Bigr)
    \geq  c_1 > 0,
\end{equation*}
for some $c_1$, uniformly on $\{T>N\}$.
The argument of  \cite{FMMV} uses Markov properties in the environment on
$[X_T + L_{i-1})$ given our realization of  $L_{i-1}$. The heavy-tailed
distribution for the renewals  ensures that for each $x > X_T + L_{i-1}$,
the corresponding renewal process has probability (uniformly in $x$ and $i$)
bounded away from zero of satisfying the emptyness condition entering in the
definition of $L_i$, and so condition (i) holds outside exponentially small
probability. We then show that the process will infect site $X_T + L_{i}$
outside exponentially small probability in $i$.
From this, we easily obtain
\begin{equation}
\label{eq5.2}
P\bigl(\{\tau = \infty\} \cap
    \{\text{$\nexists$ extreme $T$ such that $H_T$ occurs} \}
    \bigr)
    = 0.
\end{equation}
Indeed, whenever event $H_T$ does not happen we have a random
finite index $U$ such that $H_{T,U}$ is the first event $H_{T,i}$ that did not
happen. Consider the random time $S = 2^{n_U}$. By
Lemma~\ref{lema:infinite_extremes} we can
find an extreme stopping time $T_2 > S$ and once again we have
$P (H_{T_2} \mid \cA_{T_2}) \geq c_1$. Iterating this reasoning, we
deduce~\eqref{eq5.2}.

By~\eqref{eq5.2} we can conclude that a.s.\ there is some extreme
time $T$ for which $H_T$ happens. Consider the sequences
$\{L_j\}_{j=0}^{\infty}$ and $\{n_j\}_{j=0}^{\infty}$ associated with
$T$ and let $t > 2^{n_1}$. Let $i$ be the unique index such that
$2^{n_{i-1}} < t \le 2^{n_{i}}$ and define
$x_t := X_T + L_{i-2}$. By construction of event $H_T$, we have $\xi_s(x_t) = 1$
on the whole interval $[2^{n_{i-2}}, t] \supset [t/2, t]$. We estimate
$x_t$ by noticing
\begin{equation*}
x_t = X_T + L_{i-2}
    \le X_T + \smash{\sum_{j=1}^{i-2}} \, j n_0
    \le X_T + \frac{n_0}{2} (i-1)^2
    \le X_T + \frac{n_0}{2} \Bigl(\frac{\ln t}{\ln 2} - n_0\Bigr)^2
    \!\!\ll \ln^3 t,
\end{equation*}
as $t \to \infty$. This implies the corollary for finite initial configurations.
If $\sum_{x} \xi_{0} (x) = \infty$ then it is easy to see that there exists
$x$ so that taking $T=0$ and $X_T = x$ the event $H_T$ occurs (though of
course $T$ is not extreme in this case).
\end{proof}

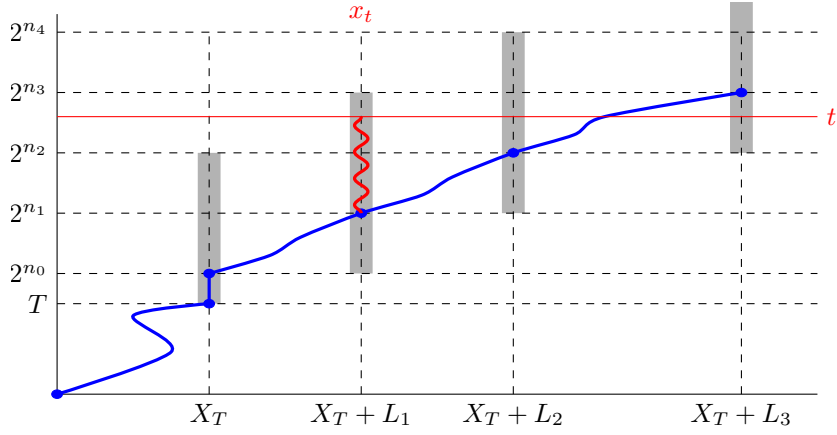
\begin{figure}
\centering
\begin{tikzpicture}[yscale=.8,
    infection/.style={very thick, blue}
    ]
    \clip (-.6,-.6) rectangle (10.3, 6.5);
    \draw (0,0) -- (10,0) (0,0) -- (0,6.3);
    % coordinates
    \coordinate (T)  at (0,1.5) node[left ] at (T)  {$T$};
    \coordinate (r0) at (0,2)   node[left ] at (r0) {$2^{n_0}$};
    \coordinate (r1) at (0,3)   node[left ] at (r1) {$2^{n_1}$};
    \coordinate (r2) at (0,4)   node[left ] at (r2) {$2^{n_2}$};
    \coordinate (r3) at (0,5)   node[left ] at (r3) {$2^{n_3}$};
    \coordinate (r4) at (0,6)   node[left ] at (r4) {$2^{n_4}$};
    \coordinate (l0) at (2,0)   node[below] at (l0) {$X_T$};
    \coordinate (l1) at (4,0)   node[below] at (l1) {$X_T+L_1$};
    \coordinate (l2) at (6,0)   node[below] at (l2) {$X_T+L_2$};
    \coordinate (l3) at (9,0)   node[below] at (l3) {$X_T+L_3$};
    % help lines
    \draw[dashed] (T) -- ++(10,0);
    \foreach \l in {0,1,2,3} \draw[dashed] (l\l) -- ++(0,6);
    \foreach \l in {0,1,2,3,4} \draw[dashed] (r\l) -- ++(10,0);
    % rectangles
    \draw[opacity=.3, line width=3mm] ($(l0) + (T)$)  -- ($(l0) + (r2)$);
    \draw[opacity=.3, line width=3mm] ($(l1) + (r0)$) -- ++(0,3);
    \draw[opacity=.3, line width=3mm] ($(l2) + (r1)$) -- ++(0,3);
    \draw[opacity=.3, line width=3mm] ($(l3) + (r2)$) -- ++(0,3);
    % infection paths
    \draw[infection] plot [smooth] coordinates
        {(0,0) (1.5,.7) (1,1.3) (2,1.5)} -- (2,2)
        plot [smooth] coordinates { (2,2) (2.8,2.3) (3.2,2.6)
                                    (4,3) (4.8,3.3) (5.2,3.6)
                                    (6,4) (6.8,4.3) (7.2,4.6) (9,5)};
    \foreach \x/\y in {0/0, 2/1.5, 2/2, 4/3, 6/4, 9/5}
        \filldraw[blue] (\x,\y) circle (2pt); % dots
    % time t and x_t
    \draw[red] (0,4.6) -- (10,4.6) node[right] {$t$};
    \draw[red, decorate, decoration=snake, very thick] (4,3) -- (4,4.6);
    \node[red] at (4, 6.3) {$x_t$};
\end{tikzpicture}
    \caption{On $H_T$ we have an infinite infected path (in blue)
    that passes through points $(X_T + L_i, 2^{n_i})$. The gray areas
    represent absence of cure marks.
    For $2^{n_{i-1}} < t \le 2^{n_i}$ we choose $x_t = X_T + L_{i-2}$,
    which ensures $\xi_s(x_t) = 1$ on the whole interval
    $[2^{n_{i-2}}, t] \supset [t/2, t]$.}
\label{fig:lasting_infection}
\end{figure}

\begin{remark}
\label{rem:H_T_high_dimension}
It should be noted that the event $H_T$ in higher dimensions involves a
direction  along one of the coordinate axes in $\ZZ^d$ away from the origin.
\end{remark}

%As a consequence of \eqref{eq5.2} we have:
%\begin{remark}
%\label{rem:lasting_infty_sites}
%If $\sum_{x} \xi_{0} (x) = \infty$ then it is easy to see that there exists
%$x$ so that taking $T=0$ and $X_T = x$ the event $H_T$ occurs, though of
%course $T$ is not extreme.
%\end{remark}

Before proving Theorem \ref{teo:complete_convergence} we need some
definitions and basic lemmas.

\begin{lema}
\label{lema:freely_infects}
Let $d \ge 1$ and $\sB(r) = [-r,r]^{d}$. Let $V_t$ be the event that
for every $x, y \in \sB(\ln^3 t)$ we have $(x, t- t^{\epsilon_1})$
freely-infects $(y,t)$ in $\sB(\ln^3 t)$, where again
$\epsilon _1$ arises from Proposition~7 of \cite{FMMV}.
Then $\lim_{t \to \infty} P(V_t)=1$.
\end{lema}

\begin{proof}
\noindent
Fix $x$ and $y$ and take a shortest path $x = x_0, x_1, \dots, x_n = y$ from
$x$ to $y$, where $x_i$ is a nearest neighbour of $x_{i - 1}$, and
$x_i \in \sB(\ln^3 t)$ for every $i$. Clearly, $n \leq C(d) \ln^3 t$ for some
positive constant $C$. Then we see that
\begin{equation}
\label{eq5.4}
P\biggl(
    \begin{array}{l}
    \text{$(x,t-t^{\epsilon_1})$ does not freely-} \\
    \text{infect $(y, t)$ in $\sB(\ln^3 t)$}
    \end{array}
\biggr)
    \le P \bigl( \Poi(\lambda t^{\epsilon_1}) \leq C(d) \ln^3 t \bigr)
    \le e^{-c\thinspace {t^{\epsilon_1}}},
\end{equation}
for some positive $c = c(\lambda)$ as $t \to \infty$, where
$\Poi(u)$ denotes a Poisson random variable of rate $u$. Thus, we can write
\begin{equation*}
P(V_t)
    \geq 1 - \sum_{\mathclap{x,y \in \sB(\ln^3 t)}} \
        P\biggl(
            \begin{array}{l}
            \text{$(x,t-t^{\epsilon_1})$ does not freely-} \\
            \text{infect $(y, t)$ in $\sB(\ln^3 t)$}
            \end{array}
        \biggr)
    \geq 1 - C(d)(\ln^3 t)^{2d} e^{-c \, t^{\epsilon_1}}. \qedhere
\end{equation*}
\end{proof}

\begin{proof}[Proof of Theorem~\ref{teo:complete_convergence}]
Notice that on $\{\tau = \infty\}$ if we also ensure the
ocurrence of events $V_t$,
\begin{align*}
    W_t &:= \bigl\{\exists x \in \sB(\ln^3 t) : \xi_s(x) = 1\ \text{on $[t/2,t]$}\bigr\},
\quad \text{and} \quad \\
    U_t &:= \bigl\{
    \cR_x \cap (t-t^{\epsilon_1}, t) = \emptyset,\,
    \forall x \in \sB(\ln^3 t)
    \bigr\}
\end{align*}
then every site of $\sB(\ln^3 t)$ is infected at time $t$. Then, we can write for
any fixed $x \in \ZZ^{d}$ that
\begin{equation*}
P(\tau=\infty, \xi_t(x)=0)
    \le P(V_t^{\comp}) + P(\{\tau=\infty\} \cap W_t^{\comp}) + P(U_t^{\comp})
\end{equation*}
for sufficiently large $t$.
Notice that all three terms on the right hand side tend to zero as
$t \to \infty$. Indeed, the first one tends to zero by
    Lemma~\ref{lema:freely_infects},
the second one by Corollary~\ref{coro:lasting_infection}, and for the third one we have by 
Proposition~7 of~\cite{FMMV} that it has probability less than
$C(d) (\ln^3 t)^{d} t^{-\epsilon_1}$.
We conclude that for any finite set $K \subset \ZZ^{d}$ we have
\begin{align*}
P(\{\tau=\infty\} \cap \{\xi_t(x) = 1, \forall x \in K\})
    &\geq P(\tau=\infty) - \smash{\sum_{x \in K}} P(\tau=\infty, \xi_t(x) = 0)
    \\
    &\to P(\tau=\infty)
\end{align*}
as $t \to \infty$ and Theorem~\ref{teo:complete_convergence} follows.
\end{proof}

%\textcolor{blue}{
\begin{remark}\label{relax}
	As brought up at the introduction, the upper bound on the tail of $\mu$
    prescribed in Condition C), may indeed be dropped as a hypothesis in
    Theorem~1 of~\cite{FMMV}. We briefly explain an alternative argument
    for that result, dispensing with the upper bound. For short, we refer to
    the notation and passages in~\cite{FMMV}.  
	
	In the argument leading up to the proof of Theorem 1 in~\cite{FMMV}, we use Proposition 7 of~\cite{FMMV},
	for which we require the upper bound; but we may instead use Corollary 4~\cite{FMMV}, for which that bound is not required.
	In that corollary, there figures an interval, denoted by $I$, and a subinterval $J\subset I$. We apply it with 
	$I=I_i=[t_02^i, t_02^{i+1}]$ and obtain $J=J_i=\big(\sigma_i,\sigma_i+(t_{0} 2^{i} )^{\gamma'}\big)$, $i=1,2,\ldots$,
	where $\gamma':=\epsilon_3/2$, with $\epsilon_3$ as in Corollary 4~\cite{FMMV}, for some $\sigma_i\in I_i$ such that $J_i\subset I_i$. 
	We then modify the argument on page 2911 of~\cite{FMMV} as follows:
	\begin{enumerate}
		\item We replace (II) by ($\text{II}^\prime$), where the interval $[t_{0} 2^{i} - (t_{0} 2^{i} )^{\gamma}, t_{0} 2^{i}]$ in (II) is replaced with
		$J_i$ in ($\text{II}^\prime$);
		\item  We next replace (III) by ($\text{III}^\prime$), where the interval 
		$$\Big(( t_{0} 2^{i} - (t_{0} 2^{i} )^{\gamma}+ \frac{k-L_{i-1}}{i\log t_0} (t_{0} 2^{i})^{\gamma}\,,\, t_{0} 2^{i}- (t_{0} 2^{i})^{\gamma} + \frac{k+1-L_{i-1}}{i\log t_0} (t_{0} 2^{i})^{\gamma})\Big)$$ 
		in (III) is replaced by
		$$\Big( \sigma_i+ \frac{k-L_{i-1}}{i\log t_0} (t_{0} 2^{i})^{\gamma'}\,,\,\sigma_i+ \frac{k+1-L_{i-1}}{i\log t_0} (t_{0} 2^{i})^{\gamma'}\Big)$$  in ($\text{III}^\prime$);
		\item We finally replace the last paragraph in the proof of Lemma 9  in~\cite{FMMV} by the following (parallel) paragraph: %\\ \vspace{-.2cm}\\
		%\smallskip
		
		 By Corollary 4 in~\cite{FMMV}, the probability of  ($\text{II}^\prime$) occurring and $L_{i} \leq L_{i-1}+ i\log t_0$ is bounded by 
		 $i\log(t_0) (t_{0} 2^{i})^{- \gamma''}$, where $\gamma'':=2\gamma'/3=\epsilon_3/3$,
		 (again supposing $t_{0}$ is large). 
		 Similarly the intersection of  ($\text{III}^\prime$) and $L_{i} \leq L_{i-1}+ i\log t_0$ has a probability bounded by
         \begin{equation*}
             i \log(t_0)e^{\smash{- \lambda(t_{0} 2^{i})^{\gamma''} /i\log(t_{0})}}.
         \end{equation*}
	\end{enumerate}
\end{remark}
\begin{remark}\label{comp}
As pointed out above, the argument in Remark~\ref{relax} dispenses with Proposition 7 of~\cite{FMMV}, relying exclusively on Corollary 4 of~\cite{FMMV} to establish Theorem 1 of~\cite{FMMV}. But to argue Theorem~\ref{teo:complete_convergence}, as we did above in this section, we require the full force of Proposition 7 of~\cite{FMMV}. Indeed, in the proof of Theorem~\ref{teo:complete_convergence} above, in the event $U_t$ we are not able to replace the {\it fixed} interval $(t-t^{\epsilon_1}, t)$, for which we may claim the absence of cure marks with high probability using Proposition 7, by an interval with similar properties provided by Corollary 4, since (the closure of) the latter interval might not contain $t$, and that would invalidate our argument.
\end{remark}
%}

\section{Closeness to determinism}
\label{sec:closeness_determinism}

In this section we consider a strengthening of
Theorem~\ref{teo:complete_convergence}. This requires greater regularity on our
renewal distribution. We require not merely that condition C)
of~\cite{FMMV} holds but that $F$ has a regular tail power:
\begin{equation*}
\bar{F}(t) \equiv 1 - F(t) \in RV(-\alpha)
\end{equation*}
for some $0 < \alpha < 1$, where $RV(\beta)$ denotes the set of functions that
for large $t$ are of the form $t^{\beta} L(t)$ for $L$ slowly varying.
If $\alpha \le 1/2$ we require, in addition, the second condition of
Theorem~1.4 of~\cite{CD} that function
\begin{equation}
\label{eq:asy_negligible}
    I_1^{+} (\delta; t)
    := \int_{1 \le z \le \delta t}\!\!  \frac{F(t - \mathrm{d}z)}{z {\bar{F}(z)}^{2}}
    \quad \text{satisfies} \quad
    \lim_{\delta \to 0}
    \varlimsup_{t \to \infty} t{\bar{F}(t)} \cdot I_1^{+}(\delta;t) = 0,
\end{equation}
which in the notation of~\cite{CD} is saying that
$I_1^{+} (\delta; t)$ is \textit{asymptotically negligible}.

Theorem \ref{teo:complete_convergence} tells us that on the event
$\{ \tau = \infty \}$ the configuration $\xi_t$ converges to
$\delta_{\underline{1}}$ in distribution as
$t \to \infty$, which is equivalent to
\begin{equation*}
\text{for every $x \in \ZZ^d$,}
\quad \quad \quad
\text{$\xi_t (x) \overset{P}{\longrightarrow} 1$ as $t \to \infty$.}
\end{equation*}

This is because the renewal (or healing) points become so sparse as $t$ becomes
large that the infection process infects all sites in a bounded region
``deterministically'' if there are no healing points nearby.

One way of expressing this is to introduce the $\sigma$-field $\cG$
generated by the renewal processes $(\cR_u)_{u \in  \ZZ^d}$ and the
extinction random variable $\tau$. We should have that, when
the infection survives, the conditional probability
$P(\xi_t (x) = 1 \mid \cG)$ should be close to $1$ for large
$t$ if there are no points of $\cR_x \cap [0, t]$ close to $t$.
Refining further, one might hope that on $\{\tau = \infty\}$
it holds
\begin{equation*}
\lim_{t \to \infty}
     \big|P(\xi_t (x) = 0 \mid \cG) - e^{-2 \lambda d Y_t (x)} \big| = 0
\quad
\text{for every $x \in \ZZ^{d}$,}
\end{equation*}
where $Y_t(x) := t - \sup \{\cR_x \cap [0, t]\}$ is the age process.
In fact this depends on the power decay of $\bar{F}$.

\begin{teo}
\label{teo:closeness}
%\label{thm3}
\noindent
If $\bar{F} \in RV(- \alpha)$ for $0 < \alpha < 1$ then
for all $x \in \ZZ^d$:

\begin{enumerate}[(i)]
\item If $\alpha < 1/2$ and also~\eqref{eq:asy_negligible}, it holds
    on $\{\tau = \infty\}$ that
    \begin{equation*}
        \lim_{t \to \infty}
         \big|P(\xi_t (x) = 0 \mid \cG) - e^{-2 \lambda d Y_t (x)} \big| = 0.
    \end{equation*}

\item If $\alpha > 1/2$ and also $F(t)>0$, for every $t>0$, it holds
    on $\{\tau = \infty\}$ that
    \begin{equation*}
    \varlimsup_{t \to \infty}
         \Bigl(P(\xi_t (x) = 0 \mid {\cG}) - e^{-2 \lambda d Y_t (x)}\Bigr) > 0.
    \end{equation*}
\end{enumerate}
\end{teo}

\begin{remark}
The case $\alpha = 1/2$ is not explicitly treated (though it is treatable)
as it depends on how $\bar{F}(t)/t^{1/2}$ behaves as $t \to \infty$.
\end{remark}

\begin{remark}
In the case $\alpha > \frac{1}{2}$ we do not need the full force of~\cite{CD},
the preceding strong renewal theorem of \cite{E} suffices.
\end{remark}

The same proof can be adapted to reach a more precise conclusion when
$\alpha \in (1/2,1)$.

\begin{teo}
\label{teo:closeness_improved}
%\label{thm4}
Assume that $\bar{F} \in RV(-\alpha)$ for $\alpha \in (1/2,1)$ and that
$F(t)>0$ for every $t>0$. For all $x \in \ZZ^{d}$:
\begin{enumerate}[(i)]
\item If $1 \le k < 2d$ and
    $1 - \alpha \in \bigl(\frac{1}{k + 2}, \frac{1}{k + 1}\bigr)$,
    then on $\{\tau = \infty\}$ we have
    \begin{equation*}
    \varlimsup_{t \to \infty}
        \Big( P(\xi_t (x) = 0 \mid \cG) - e^{-(2d - k) \lambda Y_t (x)} \Big)= 0.
    \end{equation*}

\item If $1 - \alpha < \frac{1}{2d + 1}$ then for every $0 \leq s < \infty$ we have
    on $\{\tau = \infty\}$ that
\begin{equation*}
\varlimsup_{t \to \infty}
    \I_{\{Y_t(x) =s\}} P(\xi_t (x) = 0 \mid \cG) = 1.
\end{equation*}
\end{enumerate}
\end{teo}

We provide a detailed proof for Theorem~\ref{teo:closeness}.
The same steps are used (in generalized form) for
Theorem~\ref{teo:closeness_improved} but the extra details involved do not
add any insight to the result. Considering this, we opted to only
sketch the proof of Theorem~\ref{teo:closeness_improved}, pointing out the
differences to its simpler version. We require preliminary lemmas first.

Given an integer $M$ and $t \geq 0$ we define the event $H(M,t)$ to be that
for some stopping time $T < t$ with $X_T \notin [-M,M] ^d$ the event $H_T$
occurs (see Remark~\ref{rem:H_T_high_dimension}
following Corollary~\ref{coro:lasting_infection}).
The next lemma is immediate from the argument in
Corollary~\ref{coro:lasting_infection}.
\begin{lema}
\label{asy}
For any $M \in \NN$ and finite $\xi_0$ we have that as $t \to \infty$
\begin{equation*}
P(H(M,t) \mid \cG) \xrightarrow{\text{a.s.}} \I_{\{\tau = \infty\}}.
\end{equation*}
\end{lema}

\begin{proof}{\belowdisplayskip=-12pt}
From the proof of Corollary~\ref{coro:lasting_infection} we know
$P\!\bigl(\{\tau = \infty\} \cap
    \bigl(\cup_{t \geq 1} H(M,t)\bigr)^{\comp}\bigr)=\break0$.
Since $H(M, t)$ is increasing in $t$, we have that the limit of
$P(H(M,t) \mid \cG)$ as $t \to \infty$ is almost surely
\begin{align*}
P(\cup_{t\geq 1} H(M,t) \mid \cG)
    &= P(\tau = \infty \mid \cG) - P\bigl(\{\tau = \infty\} \cap
    \bigl(\cup_{t \geq 1} H(M,t)\bigr)^{\comp} \mid \cG\bigr) \\
    &= \I_{\{\tau = \infty\}}.
\end{align*}\qedhere
\end{proof}

For the next lemma, let us define $\theta_t$ as the time-shift by $t$ of the
infection Poisson processes $\{N^{x,y}\}$. It holds
\begin{lema}
\label{lema:asy_independence}
Given $M \in \NN$ and $t_0 > 0$, let $A$ be some event generated by
Poisson processes ${N^{x,y} \cap [0,t_0]}$ for $x,y \in [-M,M] ^d$. Then,
\begin{equation*}
\lim_{t \to \infty}
    P\bigl( \theta_t(A) \mid \cG \bigr)
    = P(A) \quad \text{a.s. on $\{\tau = \infty\}$}.
\end{equation*}
\end{lema}

\begin{proof}
The conditional probability
$P\bigl( \theta_t(A) \mid \cG \bigr)$ on the event $\{\tau = \infty\}$
can be written as
\begin{equation*}
P(\theta_t (A) \mid \cG)
    = P( \theta_t (A) \cap H(M,t) \mid \cG ) +
        P( \theta_t (A) \cap H(M,t)^{\comp} \mid \cG).
\end{equation*}
We now claim that
$P\bigl(\theta_t (A) \cap  H(M,t) \mid \cG\bigr) = P(A) \cdot P(H(M,t) \mid \cG)$.
Indeed, consider the families
\begin{align*}
\cC
    &:= \bigl\{C \in \cG;\;
    P(\theta_t (A) \cap H(M,t) \cap C) = P(A) \cdot P(H(M,t) \cap C)\bigr\}, \\
\cP
    &:= \bigl\{ V \cap W;\;
    V \in \sigma(\cR_z;\; z \in \ZZ^{d}), W \in \sigma(\tau)\bigr\}.
\end{align*}
It is straightforward to check that $\cC$ is a $\lambda$-system and $\cP$ is a
$\pi$-system that generates $\cG$. Notice that $H(M,t) \subset \{\tau=\infty\}$.
If $W \supset \{\tau = \infty\}$ we have
\begin{equation*}
    P(\theta_t (A) \cap H(M,t) \cap (V \cap W))
    \!= P(\theta_t (A) \cap (H(M,t) \cap V))
    \!= P(A) \cdot P(H(M,t) \cap V),
\end{equation*}
since $A$ does not depend on renewals, only on a region of infection that is
disjoint of the one event $H(M,t) \cap V$ depends. If
$W \nsupseteq \{\tau=\infty\}$, then both sides are zero. Thus, we conclude
that $\cP \subset \cC$ and by Dynkin's $\pi$-$\lambda$ Theorem the claim follows.
The result follows from Lemma~\ref{asy}.
\end{proof}

\begin{remark}
The limit also holds a.s. on $\{\tau < \infty\}$, with a simpler proof, but
this is not needed in our argument.
\end{remark}

\begin{coro} \label{refmademe}
For $x \in \mathbb{Z}^d$ and $t > 0 $,
let
\begin{equation*}
A_t :=\{ \forall z \sim x, \ N^{z,x} \cap [t-Y_t(x),t] = \emptyset\},
\end{equation*}
where $\sim$ signifies the relation of being neighbour. Then,
\begin{equation*}
\lim_{t \to \infty}
\big\vert P\bigl(A_t \mid \cG\bigr) -e^{-2d \lambda (Y_t(x))} \big\vert
= 0 \quad \text{a.s. on $\{\tau = \infty\}$}.
\end{equation*}
\end{coro}
\begin{proof}
Fix $0 < M < \infty $ and consider events
\begin{equation*}
A^{i,M}_t = \Bigl\{ \forall z \sim x, \  N^{z,x} \cap \bigl[t - \frac{i}{M},t\bigr] = \emptyset\Bigr\}.
\end{equation*}
By Lemma \ref{lema:asy_independence}, for each $0 \leq i \leq M^2 , \  P(A^{i,M}_t \mid \cG) \rightarrow e^{-2d \lambda i/M}$ as $t $ becomes large on $\{ \tau = \infty\}$.
So by monotonicity for each $i$ and for $t$ large,  on $\{ \tau = \infty\}$
\begin{equation*}
-\frac{1}{M} \leq \I_{\{Y_t(x) \in{[\frac{i}{M} , \frac{i+1}{M}]}\}} \bigl(  e^{-2d \lambda i/M} -P(A_t \mid \cG) \bigr) \leq  \frac{3d\lambda}{M}
\end{equation*}
and similarly for $t$ large on $\{ \tau = \infty\}$
\begin{equation*}
0 \leq \I_{\{Y_t(x) \geq M\}} P(A_t \mid \cG)  \leq  e^{-d \lambda M}.
\end{equation*}
The result now follows by the arbitrariness of $M$.
\end{proof}
\begin{remark}
\label{rem:refmademe}
The argument generalizes with $A_t$ replaced by
\begin{equation*}
\{ \forall (z_i,y_i), i= 1,2,\dots r,\  N^{z_i,y_i} \cap [t-Y_t,t] = \emptyset\}.
\end{equation*}
\end{remark}

We now bring in two probability estimates. The first is a generalization of
Lemma~\ref{lema:freely_infects} and follows quickly from the bounds arrived at
in its proof.
\begin{coro}
\label{coro:bound_Cn}
Fix $\epsilon_2 \in (0,1)$.
Let $C_n = C_n (\epsilon_2)$ be the event that there exists a (time) interval
${I = [T, T + 2^{n \epsilon_2}] \subset [2^{n-1} ,2^{n+1}]}$
and sites $x, y \in \sB(n^3)$ such that $(x,T)$ does not
freely-infect $(y,T + 2^{n \epsilon_2})$ in $\sB(n^{3})$.
There exist constants $c(\lambda), K(d), n_0(\lambda, d, \epsilon_2) > 0$ such
that for all $n \geq n_0$ we have
\begin{equation*}
P(C_n)
    \le K 2^{n(1-\epsilon _2)} n^{6d} \cdot e^{-c 2^{n \epsilon_2 }}.
\end{equation*}
\end{coro}

\begin{proof}
Let $t_j := 2^{n-1} + j \cdot 2^{n \epsilon_2}/2$ and notice that intervals
$I_j = [t_j, t_{j+1}]$ for
$0 \le j \le \lfloor 3 \cdot 2^{n(1 - \epsilon_2)}\rfloor$ cover
$[2^{n-1}, 2^{n+1}]$. Moreover, if $C_n$ happens then the interval
$[T, T + 2^{n \epsilon_2}]$ obtained must contain some $I_j$. The
argument from Lemma~\ref{lema:freely_infects} shows that for any $I_j$ the
probability that there are $x,y \in \sB(n^3)$ such that $(x,t_j)$ does not
freely-infect $(y, t_{j+1})$ in $\sB(n^{3})$ is bounded by
\begin{equation*}
\sum_{\mathclap{x,y \in \sB(n^3)}}
    P\Bigl(
        \Poi\bigl(\lambda \cdot (2^{n \epsilon_2}/2)\bigr)
        \le C(d) n^{3}
    \Bigr)
    \le K(d) n^{6d} \cdot e^{- c(\lambda) 2^{n \epsilon_2}}
\end{equation*}
for positive constants $K(d)$ and $c(\lambda)$.
The result follows from union bound.
\end{proof}

\begin{lema}
\label{lema:bound_Bn}
Let $\alpha < 1/2$, $\bar{F} \in RV(-\alpha)$
satisfying~\eqref{eq:asy_negligible}, and fix
$\epsilon \in (0, \tfrac{1}{2} - \alpha)$. The event ${B_n = B_n(\epsilon)}$
defined by
\begin{equation*}
B_n
    := \Bigl\{
    \begin{array}{c}
    \exists\ \text{distinct}\ z, z' \in \sB(n^3), \
    s  \in [2^{n-1}, 2^{n+1}]\ \text{such that} \\
    \cR_z \cap [s,s+1] \neq \emptyset, \
    \cR_{z'} \cap [s, s + 2 \cdot 2^{n \epsilon }] \neq \emptyset
    \end{array}
    \Bigr\}
\end{equation*}
satisfies
\begin{equation*}
    P(B_n) < K \cdot n^{6d} \cdot 2^{-n(1- 2 \alpha  -2 \epsilon )}
\end{equation*}
for a positive constant $K=K(\alpha, d)$ and sufficiently large $n$.
\end{lema}

\begin{proof}
We simply write event $B_n$ as the union of $B_n(z, z')$
for $z, z' \in \sB(n^3)$, where
\begin{equation*}
B_n(z, z')
    := \{\exists s \in [2^{n-1}, 2^{n+1}]:  \cR_z \cap [s,s+1] \neq \emptyset, \
    \cR_{z'} \cap [s,s+2 \cdot 2^{n \epsilon}]  \neq \emptyset\}.
\end{equation*}
We then note that $B_n(z, z')$ is in turn a subset of the union
\begin{equation*}
\bigcup_{\smash{j = 2^{n-1}}}^{\smash{2^{n+1}}}
    \{
        \cR_z \cap [j,j+2] \neq \emptyset, \
        \cR_{z'} \cap [j,j + 2 \cdot 2^{n \epsilon } + 1] \neq \emptyset
    \},
\end{equation*}
whose events for fixed $z$, $z'$ and $j$ will be denoted $B_n(z,z',j)$.
By independence, since $z \neq z'$ we have
\begin{equation*}
    P\bigl(B_n(z,z',j)\bigr)
    = P\bigl(\cR_z \cap [j,j+2] \neq \emptyset\bigr) \cdot
    P(\cR_{z'} \cap [j,j + 2 \cdot 2^{n \epsilon } + 1] \neq \emptyset).
\end{equation*}

Let us assume that $\cR$ is a non-arithmetic renewal process.
The Strong Renewal Theorem (Theorem~1.4 of~\cite{CD})
provides an estimate
\begin{equation*}
    P\bigl(\cR \cap [j,j+2] \neq \emptyset\bigr)
    \le U([j,j+2])
    \sim C(\alpha) \frac{L(j)}{j^{1-\alpha}}
    \quad \text{as $j \to \infty$,}
\end{equation*}
where $U$ denotes the renewal measure associated to $F$, $L$ is a
slowly varying function, and $C(\alpha)$ is a positive constant.
Also, the definition of slowly varying function implies the bounds
\begin{align*}
P\bigl(\cR \cap [j,j+2] \neq \emptyset\bigr)
    &\ll 2^{-n(1-\alpha-\epsilon/2)}, \quad \text{and} \\
P\bigl(\cR \cap [j,j+2 \cdot 2^{n \epsilon} + 1] \neq \emptyset\bigr)
    &\le \sum_{k=j}^{j+2 \cdot 2^{n \epsilon} - 1}
        P\bigl(\cR \cap [k,k+2] \neq \emptyset\bigr) \\
    &\ll (2 \cdot 2^{n \epsilon}) \cdot 2^{-n(1-\alpha-\epsilon/2)}.
\end{align*}
The result now follows from the usual union bound,
at least for the non-arithmetic case. For the arithmetic case, we just
have to consider intervals $[j, j+h]$ for $h$ being the span of $\cR$ and
the same reasoning applies.
\end{proof}

Finally, the following estimate, a result which is similar to Lemma~3
of~\cite{FMMV}, shows that even in the case in which there are renewal
marks on some interval $[2^{n-1}, 2^{n+1}]$, the
probability that these marks
are too dense on this interval decays rapidly with $n$.

\begin{lema}
\label{lema:bound_Dn}
Fix $\alpha, \epsilon \in (0,1)$.
There is $g(\alpha) \in (0,1)$ such that the event
$D_n = D_n(\epsilon)$ defined by
\begin{equation*}
D_n :=
    \{ \exists z \in \sB(n^3), \ I \subset [2^{n-1}, 2^{n+1}]:
    \vert I \vert = 2^{n \epsilon}, \
    \vert \cR_z \cap I  \vert \geq n^2 2^{n \epsilon g(\alpha)}
    \}
\end{equation*}
satisfies
\begin{equation*}
    P(D_n) \le K(d) n^{3d} \cdot 2^n \cdot 2^{- c \epsilon^2 n^2}
\end{equation*}
for constants $c>0$ and $K(d) > 0$ and sufficienly large $n$.
\end{lema}

\begin{proof}
Consider the collection of intervals
$I_j = [2^{n-1} + j, 2^{n-1} + j + 2^{n \epsilon} + 1]$ for integer $j$ satisfying
$0 \le j \le 3 \cdot 2^{n-1}$. Then $[2^{n-1}, 2^{n+1}] \subset \cup_j I_j$ and
whenever event $D_n(\epsilon)$ happens the interval $I$ obtained must be
contained in some $I_j$ and implies there are many renewal marks inside $I_j$.
Denoting $|I_j| = 2^{n \epsilon} + 1$ by $l$, the proof of Lemma~3
of~\cite{FMMV} gives the following estimate
\begin{equation}
\label{eq:lemma_3_FMMV}
P(|\cR \cap I_j| \geq l^{1 - \epsilon_3} \ln^2 l)
    \le 2 \cdot e^{- \ln^2 l}
    \le 2^{- c \epsilon^2 n^2}
    \qquad \text{for large $n$},
\end{equation}
where constant $\epsilon_3 > 0$ satisfies $t^{-(1-\epsilon_3)} \le \bar{F}(t)$
for large $t$ (the proof of Lemma~3 of~\cite{FMMV} only uses the lower
bound of condition C)\ ).  Since $\bar{F}(t) \in RV(-\alpha)$ and $\alpha
\in (0,1)$, we can take $\epsilon_3 := (1 - \alpha)/2$.
Let us define $g(\alpha) := 1 - \epsilon_3/2$, so that $g(\alpha) > 1 -
\epsilon_3$. It is straightforward to check that
\begin{equation*}
n^2 2^{n \epsilon g(\alpha)}
    \gg l^{1 - \epsilon_3} \ln^2 l
\qquad \text{as $n \to \infty$}.
\end{equation*}
Using~\eqref{eq:lemma_3_FMMV} we conclude that
\begin{equation*}
P(D_n)
    \le \sum_{z \in \sB(n^{3})} \sum_{j=0}^{3 \cdot 2^{n-1}}
        P\bigl(|\cR \cap I_j| \geq n^2 2^{n \epsilon g(\alpha)}\bigr)
    \le K(d) n^{3d} \cdot 2^{n} \cdot 2^{- c \epsilon^2 n^2}. \qedhere
\end{equation*}
\end{proof}

\begin{proof}[Proof of Theorem \ref{teo:closeness}, part (i)]
We assume without loss of generality that $x$ is the origin and denote
$Y_t(0)$ simply by $Y_t$ and recall that our estimates hold a.s.\ on the event
$\{\tau=\infty\}$. For $t > 0$ define
$n = n(t) := \lfloor \log_{2} t \rfloor$, so that $t \in [2^n,2^{n+1})$.
Fix $\epsilon \in (0, 1/2 - \alpha)$ and consider events $B_k(\epsilon)$
and $D_k(\epsilon)$, which are both $\cG$-measurable. 
These events can be used to ensure that as $t \to \infty$ the renewal marks
near $\{0\}\times \{t\}$ are relatively sparse, almost surely. Indeed,
by Lemmas~\ref{lema:bound_Bn} and \ref{lema:bound_Dn}, we have
\begin{equation*}
\sum_{k\geq 1} P(B_k \cup D_k) < \infty,
\quad \text{implying that} \quad
    \I_{B_k^{\comp} \cap D_k^{\comp}} \xrightarrow{\text{a.s.}} 1
    \quad \text{as $k \to \infty$}.
\end{equation*}
On the event $G_n := B_{n}^{\comp} \cap D_{n}^{\comp}$
there is at most one site $z \in \sB(n^3)$ with
$\cR_z \cap [t- 2^{n \epsilon}, t] \neq \emptyset$,
since otherwise event $B_{n}$ happens.
Moreover, on $D_{n}^{\comp}$ we must have some interval
$I \subset [t - 2^{n \epsilon}, t]$ that has
no cure marks of $\cR_z$ with length
\begin{equation}
\label{eq:I_free_renewals}
|I|
    \geq \frac{2^{n \epsilon}}{n^2 2^{n \epsilon g(\alpha) }}
    = \frac{1}{n^{2}} \cdot 2^{n \epsilon (1-g(\alpha))}
    \gg 2^{n \epsilon'}
    \qquad \text{as $t \to \infty$,}
\end{equation}
for $\epsilon' := \epsilon (1-g(\alpha))/2$. This implies
$\sB(n^3) \times I$ is free of renewals.

Event $G_n \in \cG$ provides some control on the renewal structure of
$\sB(n^3) \times [2^{n-1}, 2^{n+1}]$. Now, we discuss two other events that are
not $\cG$ measurable but will help us handle the infections in $\sB(n^3)
\times [2^{n-1}, 2^{n+1}]$. The first event is $C_n = C_n(\epsilon')$.
By the estimates for $P(C_n)$ from Corollary~\ref{coro:bound_Cn} and
the Borel-Cantelli Lemma we have that
\begin{align*}
P(\varlimsup_n C_n) = 0 \quad
    &\text{implies} \quad
P(\varlimsup_n C_n \mid \cG) = 0 \quad \text{a.s.\ }, \quad \\
    &\text{and hence} \quad
\lim_n P(C_n \mid \cG) = 0 \quad \text{a.s.\ }
\end{align*}
by Fatou's Lemma. The second event is denoted $H_n$ and defined as
\begin{equation*}
H_n
    := \{ \forall \ m \geq n, \  \exists \ y \in
    \sB(m^3)\setminus \sB(2) \mbox{ with }
        \xi _s(y)=1,\  \forall \ s \in [2^m,2^{m+1} ]\}.
\end{equation*}
By Corollary~\ref{coro:lasting_infection} and monotone convergence for
conditional expectations, we have $P(H_n \mid \cG) \rightarrow 1$ on
event $\{\tau = \infty\}$.

We consider $P(\xi_t(0) =1  \mid \cG)$ on survival and on event $G_n$.  There are two cases.
We first suppose that $Y_t  \leq 2^{n \epsilon }$.  The fact that the renewal environment belongs to $G_n$
implies that there are no renewals on $\sB(n^3) \times (t-Y_t,t]$.
Obviously, the event $\{\xi_t(0) =0 \} $ contains the event
\begin{equation*}
A_t = \{ \forall z \sim 0, \ N^{z,0} \cap [t-Y_t(0),t] = \emptyset\} \quad \text{(recall Corollary~\ref{refmademe})}.
\end{equation*}
Also, on $G_n$ we have
\begin{equation*}
H_{n-1} \cap C_n^{\comp} \cap A_t^{\comp} \subset \{\xi_t(0) = 1\}.
\end{equation*}
Indeed, if $2^{n \epsilon'} \le Y_t \le 2^{n \epsilon}$ then the infected site
given by $H_{n-1}$ will infect the origin on event $C_n^{\comp}$. If $Y_t < 2^{n
\epsilon'}$ then by our discussion next to~\eqref{eq:I_free_renewals} we can
find an interval $I \subset (t-2^{n \epsilon},t-Y_t)$ such that the infection
provided by $H_{n-1}$ will spread throughout $I$ and guarantee that all
neighbours of the origin, which we denote by $\Gamma_0$, will be infected
at time $\sup I \le t - Y_t$. Moreover, on $G_n$ the only cure mark in
$[t-2^{n \epsilon},t]$ is the one at the origin at time $t-Y_t$, implying that
$\Gamma_0$ is infected at time $t-Y_t$. Hence, the transmission given by
$A_t^{\comp}$ will infect the origin and we can write
$\{\xi_t(0) =0 \}  \backslash A_t \subset H^{\comp}_{n-1} \cup C_n$, implying
\begin{equation*}
P(A_t  \mid \cG)
    \le P(\xi_t(0) =0 \mid \cG)
    \le P(A_t \mid \cG) + P(H_{n-1}^{\comp} \mid \cG)+ P(C_n \mid \cG).
\end{equation*}
By Corollary~\ref{refmademe}, on survival and $G_n$ we have
\begin{equation*}
\I_{\{Y_t \leq 2^{n \epsilon}\}} \big\vert P(\xi_t(0) =0  \mid   \cG) -
    e^{-2d\lambda \cdot Y_t} \big\vert
\rightarrow 0
\qquad \text{as $t \rightarrow \infty $.}
\end{equation*}
If $Y_t \geq 2^{n \epsilon}$, then the interval $I
\subset [t-Y_t,t]$ provided by~\eqref{eq:I_free_renewals} is long enough for
event $C_n^{\comp}$ to infect the origin. Thus, on $Y_t \ge 2^{n \epsilon}$ we
have
\begin{align*}
    0 \le P(\xi_t(0) =0  \mid   \cG)  &\le P(H_{n-1}^{\comp} \mid   \cG)+ P(C_n \mid
    \cG), \\
    0 \le e^{-2d\lambda \cdot Y_t} &\le e^{-2d\lambda 2^{n \epsilon}}
\end{align*}
implying that $|P(\xi_t(0) =0  \mid   \cG) - \smash{e^{-2d\lambda \cdot Y_t}}| \to 0$
as $t \to \infty.$
\end{proof}

Now, we turn to the proof of Theorem~\ref{teo:closeness} (ii) and fix 
$\alpha  > 1/2$. We rely on two preliminary results.
Given $\epsilon > 0$ and $z \in \ZZ^d$ we say time interval $I$ is an
$\epsilon$-block (for $\cR_z$)  if $I \backslash \cR_z$ contains only
intervals of length less than $\epsilon$.

In order to motivate our next proposition, we prove:
\begin{lema}
\label{lema:eps_block_estimate}
Given $M, \delta >0$, there is $\epsilon = \epsilon(d, M, \delta, \lambda) > 0$
so that for $s$ sufficiently large and $z$ a fixed site
\begin{equation*}
P(\text{$z$ infects a neighbour in $[s,s+M]$} \mid \cG)
    < \delta
\end{equation*}
on the event where $[s,s+M]$ is an $\epsilon$-block (for $\cR_z$).
\end{lema}

\begin{proof}
Write $I_0, I_1, \ldots, I_K$ for the (ordered) intervals of
$[s, s+M] \backslash \cR_z$.
Denote by $N^z$ (resp. $N^z_i$) the union of Poisson processes $N^{y,z}$
(resp. $N^{z,y}$) with $y \sim z$. We note that event
$\{\text{$z$ infects a neighbour in $[s,s+M]$}\}$ is contained on
\begin{equation*}
\{ N^z \cap I_0 \ne \emptyset\}
\cup \cup_{j=1}^K \{ \text{$I_j$ contains points in $N^z_i$ and $N^z$} \}.
\end{equation*}
Consider event
\begin{equation*}
D
    := \Bigl\{
        \begin{array}{c}
        \forall r \in [0,M] \cap N^{z},\
        \forall u \in [0,M] \cap N_i^{z}\\
        \text{we have $|r-u| > \epsilon$, and $r > \epsilon$ and $u > \epsilon$}
        \end{array}
        \Bigr\}.
\end{equation*}
On event $[s,s+M]$ is an $\epsilon$-block we have that
\begin{equation*}
P(\text{$z$ infects a neighbour in $[s,s+M]$} \mid \cG)
    \le P(\theta_s(D^{\comp}) \mid \cG) \to P(D^{\comp})
    \; \text{as $s \to \infty$}
\end{equation*}
by Lemma~\ref{lema:asy_independence}. Now, we simply notice that fixed
$\lambda, d, M > 0$ the probability on the right hand side is
continuous function of $\epsilon$ and converges to $0$ as $\epsilon \downarrow
0$.
\end{proof}

Let us fix $z$ a neighbour of $0$. We define event $A^{n}_{M,\epsilon}$
to be the event that in $[2^n, 2^{n+1})$
there exists $t$ such that $[t,t+1] \cap \cR_0 \neq \emptyset$,
$[t+1, t+M+1] \cap \cR_0 = \emptyset$ and $[t,t+M+1]$ is an
$\epsilon$-block (for $\cR_z$). Our following result will
use the following notation for comparing sequences: we say that
$f \asymp g$ if there is $K \ge 1$ such that
$(1/K) |g(n)| \le |f(n)| \le K |g(n)|$ for every $n\ge 1$. 

\begin{prop}
\label{prop:eps_block_as}
Let $F$ satisfy the conditions of Theorem~\ref{teo:closeness} with $\alpha > 1/2$.
For $\epsilon > 0, M < \infty$,
\begin{equation*}
P(\varlimsup_n A^{n}_{M,\epsilon})
    = 1.
\end{equation*}
\end{prop}

\begin{proof}
It is based on a second moment argument. We assume $\epsilon < 1$. Consider events
$A_j = A_j(M, \epsilon)$
defined by
\begin{equation*}
    A_j :=
    \Bigl\{
    \begin{array}{c}
        \text{
        $[j,j+1] \cap \cR_0 \neq \emptyset$, \
        $[j+1, j+M+1] \cap \cR_0 = \emptyset$},\\ 
        \text{and $[j,j+M+1]$ is an $\epsilon$-block for $\cR_z$}
    \end{array}
    \Bigr\}
\end{equation*}
and for $n \geq 1$ define the random variables
\begin{equation*}
    X_n := \sum_{j=2^n}^{2^{n+1}-1} \I_{A_j}
\end{equation*}
that count the number of occurrences of events $A_j$ for
$2^{n} \le j < 2^{n+1}$. Clearly, the event $A^n_{M, \epsilon}$ contains
the event $\{X_n > 0\}$.

The largest part of the proof consists of showing the existence of a
$\delta>0$ independent of $n$ so that $P(X_n >0)>\delta$ for all $n \geq 1$.
Once we have this, we simply note that the desired conclusion follows from
Hewitt-Savage's $0$--$1$ law, considering that $\smash{\varlimsup_n A_{M,\epsilon}^{n}}$
is invariant with respect to finite permutations of the family of iid. random
variables $\{(T_i^{0}, T_i^{z}); i\ge 0\}$.

By Paley-Zygmund inequality, it suffices to find
$K = K(M, \epsilon) < \infty$ so that for $n$ large
\begin{equation*}
    E X_n^2 \leq K (E X_n)^2.
\end{equation*}

The Strong Renewal Theorem of \cite{CD} will play a key role in the bounding of
    both moments. This states (in our context, recalling that $F(t)>0$ for
    $t>0$ implies that our renewal process is non-arithmetic) that as $x$ becomes large
\begin{equation}
\label{eq:srt_U}
    U(x,x+h)x^{1-\alpha}L(x) \rightarrow c_{\alpha} h,
\end{equation}
where $U(I) := E(| \cR \cap I|)$ and $c_{\alpha}$ is a
positive constant. Notice that for all intervals
$I = [x, x+h]$ with $0 < h \le 1$ we have that $U(I)$ is comparable to
$P(\cR \cap I \neq \emptyset)$. Indeed, by Markov inequality we have
\begin{equation*}
P(\cR \cap I \neq \emptyset)
    = P\bigl(|\cR \cap I| \geq 1\bigr)
    \le U(I).
\end{equation*}
On the other hand, we have
\begin{align*}
U(x,x+h)
    = \sum_{j\geq 1} P(|\cR \cap I| \geq j)
    &\le \sum_{j\geq 1} P(|\cR \cap I| \geq 1) P(T \le h)^{j-1} \\
    &= \frac{P(|\cR \cap I| \geq 1)}{\bar{F}(h)},
\end{align*}
where we recall $T \distr \mu$ and $\bar{F}(t)>0$ for any $t>0$. This leads to
the estimate
\begin{equation*}
P\bigl(\cR \cap [x,x+h] \neq \emptyset\bigr)
    \le U(x,x+h)
    \le \bar{F}(1)^{-1} \cdot P\bigl(\cR \cap [x,x+h] \neq \emptyset\bigr).
\end{equation*}

We now show that $P(A_j)$ is comparable to $U(j,j+1)^2$ by
decomposing $P(A_j)$ with respect to what happens at the origin and at $z$.

It is immediate that
\begin{equation*}
    P(A_j)
    \le P(\cR_0 \cap [j,j+1] \neq \emptyset)
    P(\cR_z \cap [j,j+\epsilon] \neq \emptyset)
    \le U(j,j+1)^2,
\end{equation*}
For a lower bound, we have that
\begin{equation*}
P\bigl(\cR_0 \cap [j+1,j+M+1] = \emptyset,\ 
        \cR_0 \cap [j,j+1] \neq \emptyset\bigr)
    \ge K \cdot U(j,j+1) \cdot \bar{F}(M+2)
\end{equation*}
We claim that
$P([j,j+M+1] \text{ is an $\epsilon$-block})$ satisfies a similar lower bound,
for some constant ${K = K(M, \epsilon)}$. Indeed, notice that we can find
$\eta = \eta(\epsilon) > 0$ such that $F(\epsilon) > F(\eta) > 0$. If we have
$\cR_z \cap [j,j+\epsilon] \neq \emptyset$ and the
next $\lceil (M+1)/\eta \rceil$ random variables $T_i$ of the renewal process
satisfy $T_i \in [\eta, \epsilon]$ we will have an $\epsilon$-block, which
leads to the bound
\begin{equation*}
P([j, j+M+1] \text{ is an $\epsilon$-block})
    \ge U(j,j+\epsilon) \cdot (F(\epsilon) - F(\eta))^{\lceil (M+1)/\eta \rceil}.
\end{equation*}
These estimates imply that $P(A_j) \asymp U(j,j+1)^{2}$ for some 
constant $K(M, \epsilon)$. Using the estimate given by the
Strong Renewal Theorem \eqref{eq:srt_U}, defining $n = n(j)$ as the only
integer satisfying $2^{n} \le j < 2^{n+1}$ we have
\begin{equation*}
    P(A_j)
    \asymp \Bigl(\frac{c_{\alpha}}{L(j) j^{1-\alpha}}\Bigr)^{2}
    =  \Bigl(\frac{c_\alpha}{L(2^{n}) 2^{n(1-\alpha)}} \cdot
    \frac{L(2^{n}) 2^{n(1-\alpha)}}{L(j) j^{1-\alpha}}\Bigr)^{2}
    \asymp L(2^{n})^{-2} 2^{2n(\alpha-1)}.
\end{equation*}
Thus, $E X_n$ satisfies
\begin{equation*}
E X_n
    = \sum_{j=2^{n}}^{2^{n+1}-1} P(A_j)
    \asymp L(2^n)^{-2} 2^{(2 \alpha -1)n},
\end{equation*}
which by our assumption on $\alpha$ tends to infinity as $n$ becomes large.
To finish the proof we must show that $E X_n^{2}$ has an upper
bound of the same order of magnitude as $(E X_n)^{2}$. While proving first
moment estimates, we concluded that $P(A_j)$ is comparable to the
probability of the event
\begin{equation*}
A'_j
    := \{[j,j+1] \cap \cR_0 \neq \emptyset \} \cap
        \{[j,j+1] \cap \cR_z \neq \emptyset\}.
\end{equation*}
The same argument shows that $P(A_j \cap A_k) \le P(A'_j \cap A'_k)$, so it
suffices to give an appropriate upper bound to
\begin{equation*}
E \Bigl[\Bigl(\sum_{j=2^n}^{2^{n+1}-1} \I_{A'_j}\Bigr)^{2}\Bigr]
    = 2 \sum_{2^{n} \le j < k < 2^{n+1}} P(A'_j \cap A'_k) +
        \sum_{j=2^n}^{2^{n+1}-1} P(A'_j).
\end{equation*}
Our analysis rests on bounding $P(A'_k \mid A'_j)$. We note that for $j<k$
an application of the Markov property on the first renewal inside $[j, j+1]$
implies
\begin{equation*}
\inf_{x \in [k-j-1,k-j]} U(x,x+1)^2 / K^2
    \le P( A'_k | A'_j)
    \le \sup_{x \in [k-j-1,k-j]} U(x,x+1)^2
\end{equation*}
for some positive $K(\epsilon, M)$. In particular, an upper bound on
$P(A'_k \mid A'_j)$ will follow from bounding
\begin{equation*}
    C_{r} := \sup_{r-1 \le x \le r} U(x,x+1)^2
    \quad \text{for $r = k-j$}.
\end{equation*}

Let $\nu = \nu(\alpha) > 1$ be a fixed constant whose precise value we will
determine later. We fix ${M' \geq M+1}$ so that whenever $x \geq M'$ we have
in addition that
\begin{equation*}
\bigl\{\tfrac{L(y)}{L(x)}: x \le y \le 4x\bigr\} \cup
\bigl\{U(x,x+1) x^{1-\alpha}\tfrac{L(x)}{c_{\alpha}}\bigr\} \cup
\bigl\{\tfrac{U(x,x+1)}{U(x',x'+1)}: |x-x'| \le 1\bigr\}
    \!\subset\! (\tfrac{1}{\nu}, \nu).
\end{equation*}
\begin{alignat*}{3}
\text{Then for $r \geq M'$: \qquad}
    &&U(r,r+1)^2 \nu^{-2}
    &\le C_{r}
    &&\leq U(r,r+1)^2 \nu^{2}, \\
\text{and so: \qquad}
    &&c_{\alpha}^2 r^{-2(1-\alpha)} L(r)^{-2} \nu^{-4}
    &\le C_{r}
    &&\le c_{\alpha}^2 r^{-2(1-\alpha)} L(r)^{-2} \nu^{4}.
\end{alignat*}
Once again, our choice of $M'$ yields that for $r \geq M'$
\begin{equation}
\label{eq:relating_Cr}
\smash{\frac{C_{2r-1} + C_{2r}}{C_r}}
    \ge \smash{\frac{C_{2r}}{C_r}}
    \ge \nu^{-8}  \Bigl(\frac{L(r)}{L(2r)}\Bigr)^{2} 2^{2(\alpha - 1)}
    \ge \nu^{-10} 2^{2(\alpha - 1)}.
\end{equation}
Notice that for any fixed $j \in [2^{n}, 2^{n+1})$, we have
\begin{equation}
\label{eq:sum_Cr}
\sum _{k=j}^{2^{n+1}-1} P(A'_k \mid A'_j)
    \le M'+1 + \sum_{M'+1}^{2^{n+1}-1} C_r
    \le M'+1 + \sum_{l=1}^R \sum_{r \in J_l} C_r
\end{equation}
where $J_l := (M'2^{l-1}, M'2^l]$ and $R := \inf \{l:M'2^l \geq 2^{n+1}\}$.
The bound on~\eqref{eq:relating_Cr} implies
\begin{equation*}
\sum_{r \in J_l} C_r
    \le \nu^{10} 2^{-2(\alpha-1)} \sum_{r \in J_{l+1}} C_r,
    \quad \text{for any $1 \le l < R$}.
\end{equation*}
Choosing $\nu>1$ so that $q := \nu^{10} 2^{-2(\alpha-1)} < 1$, we have
from~\eqref{eq:sum_Cr} that
\begin{equation*}
\sum _{k=j}^{\mathclap{2^{n+1}-1}} P(A'_k \mid A'_j)
    \le M'+1 + (1 + q + \ldots + q^{R-1}) \sum_{\mathclap{r \in J_R}} C_r
    \le M'+1 + (1-q)^{-1} \sum_{\mathclap{r \in J_R}}C_r.
\end{equation*}
Since $J_R$ has at most $4 \cdot 2^{n}$ integer points and our conditions for
$M'$ ensure that each $C_r$, for $r \in J_R$, is comparable to one another, we
conclude that
\begin{equation*}
\sum _{k=j}^{2^{n+1}-1} P(A'_k \mid A'_j)
    \le K 2^{n} C_{2^{n+1}}
    \le K L(2^{n})^{-2} 2^{(2\alpha -1)n}
\end{equation*}
for some positive $K(\alpha)$ and the proof is completed.
\end{proof}

\begin{proof}[Proof of Theorem \ref{teo:closeness}, part (ii)]
We fix $M=1$ and postpone the definition of $\delta = \delta(d, \lambda)> 0$ and
$\epsilon = \epsilon(\delta) > 0$ that provide a suitable choice of
event $A := \varlimsup_n A^{n}_{1,\epsilon}$.
By Proposition~\ref{prop:eps_block_as}, the event $A$ occurs a.s. for any
choice of $\epsilon>0$. On $A$ we can find arbitrarily large times $t$ such that
\begin{equation*}
    \cR_0 \cap [t,t+1] \neq \emptyset,
    \quad \cR_0 \cap [t+1,t+2] = \emptyset,
    \quad \text{and} \quad
    \text{$[t,t+2]$ is $\epsilon$-block for $\cR_z$}.
\end{equation*}
The above property ensures that $Y_{t+2} \in [1,2]$. Recall that $\Gamma_0$
denotes the neighbours of the origin. Using Lemma~\ref{lema:eps_block_estimate} we have that
on $A \cap \{\tau = \infty\}$ 
\begin{align*}
P(\xi_{t+2}(0) = 1 \mid \cG)
    &\le \delta +
        P\bigl( \cup_{y \in \Gamma_0\setminus\{z\}}
            \{N^{y,0} \cap [t+2 - Y_{t+2}, t+2] \neq \emptyset\} \mid \cG\bigr) \\
    &\le \delta + \bigl(1 - e^{-(2d-1)\lambda Y_{t+2}} + \delta\bigr)
\end{align*}
for a suitable time $t$, where the last inequality follows from
Corollary~\ref{refmademe} and Remark~\ref{rem:refmademe}
when $t$ is sufficiently large. Hence, defining
\begin{equation*}
\eta(d, \lambda)
    := \inf_{x \in [1,2]} \bigl(e^{-(2d-1)\lambda x} - e^{-2d\lambda x}\bigr)
\end{equation*}
we can estimate
\begin{equation*}
\bigl(1 - e^{-2d\lambda Y_{t+2}}\bigr) - P(\xi_{t+2}(0) = 1 \mid \cG)
    \ge \bigl(e^{-(2d-1)\lambda Y_{t+2}} - e^{-2d\lambda Y_{t+2}}\bigr)
        - 2\delta
    \ge \eta - 2\delta,
\end{equation*}
which is positive once we define $\delta := \eta/4$. The choice of $\epsilon$
is made accordingly, using Lemma~\ref{lema:eps_block_estimate}.
\end{proof}

The proof of Theorem~\ref{teo:closeness_improved} follows the same lines of the
proof of Theorem~\ref{teo:closeness}. Instead of writing down every detail for
this similar proof, we give a sketch of the argument, singling out the main
differences.

\begin{proof}[Sketch of proof of Theorem~\ref{teo:closeness_improved}]
The proof is equivalent to showing that
\begin{enumerate}[a)]
\item
If $1-\alpha > \frac{1}{k+2} $ then
$\varlimsup\limits_{t}\left( P(\xi_t(0) =0 \mid \cG ) - e^{-(2d-k)\lambda Y_t }
        \right) \leq 0$.
\item
If $1-\alpha < \frac{1}{k+1} $ then for each $s > 0 $
\begin{equation*}
\smash{\varlimsup_{t} \I_{\{Y_t = s\}}}
    \bigl( P(\xi_t(0) =0 \mid \cG ) - e^{-(2d-k)\lambda Y_t }  \bigr)
    \ge 0.
\end{equation*}
\end{enumerate}
We consider first b), that is $1-\alpha < \frac{1}{k+1} $.
For $0 < \epsilon, M < \infty $ and
$z_1, \ldots, z_k \in \Gamma_0$ distinct, we  define
$A^{n}_{M,\epsilon}(k)$ to be the event that
\begin{equation*}
A^{n}_{M,\epsilon}(k)
    \!:= \!\biggl\{\!\!\!
        \begin{array}{l}
            \text{$\exists t \in [2^{n},2^{n+1});\; [t,t+1] \cap \cR_0\neq
            \emptyset$, $[t\!+\!1, t\!+\!M\!+\!1] \cap \cR_0 = \emptyset$,}
             \\
            \text{and $[t,t+M+1]$ is an $\epsilon$-block for $\cR_{z_j}$,
            for every $1 \le j \le k$}
        \end{array}
    \!\!\biggr\}.
\end{equation*}
The next claim follows Proposition~\ref{prop:eps_block_as} closely:
\begin{claim}
For any $M < \infty$ and $\epsilon > 0$ we have
    $P(\varlimsup_n A^{n}_{M,\epsilon}(k)) = 1$.
\end{claim}

\begin{proof}[Sketch of proof]
Proposition~\ref{prop:eps_block_as} is the claim with $k=1$. For $k\ge 2$,
introduce random a variable $X_n = \sum_{i=2^n}^{2^{n+1}-1}\I_{G(i)}$ where
\begin{equation*}
G(i) = \biggl\{
    \begin{array}{c}
        [i,i+1] \cap \cR_0 \neq \emptyset, [i+1, i+M+1] \cap \cR_0 = \emptyset
        \ \text{and} \\
        \text{$[i,i+M+1]$ is an $\epsilon$-block for $\cR_{z_j}$, for every $1 \le j \le k$}
    \end{array}
 \biggr\}.
\end{equation*}
So, $X_n > 0$ implies that event $A^{n}_{M,\epsilon}(k)$ occurs.
As in Proposition~\ref{prop:eps_block_as}, for all $n$ large and
$2^{n} \le i < 2^{n+1}$ we have that
\begin{align*}
    && P(G(i)) &\asymp L(2^{n})^{-(k+1)}2^{-n(k+1)(1-\alpha)}, \\
    \text{so that} &&
    E(X_n) &\asymp L(2^{n})^{-(k+1)} 2^{n(1-(k+1)(1-\alpha))}.&&
\end{align*}
Also, for $2^n \leq j-i <2^{n+1}-1$ we have that
\begin{equation*}
    P(G(j) \cap G(i))
    \le P (G(i)) \cdot \ \sup_{\mathclap{x \in [j-i-1,j-i]}}\ U(x, x+1)^{k+1}
\end{equation*}
and similar computations to Proposition~\ref{prop:eps_block_as} lead to
\begin{equation*}
\sum _{j=i}^{2^{n+1}-1} P(G(j) \mid G(i))
    \le K L(2^{n})^{-(k+1)} 2^{n(1-(k+1)(1-\alpha))}
\end{equation*}
for universal $K$ (depending on $\epsilon $ and $M$). This suffices to bound
$E(X_n^2) $ by a universal multiple of $(E X_n)^2$ and leads to the proof of
the claim.
\end{proof}

Hence, fixing $M,\delta>0$ we can follow the proof of
Theorem~\ref{teo:closeness}.(ii) (with a slight generalization of
Lemma~\ref{lema:eps_block_estimate}) to show that for any $s \in [0,M] > 0$
there are infinitely many $t_n$ tending to infinity so that $Y_{t_n} =s$ and
\begin{equation*}
         P(\xi_{t_n} (0) = 0 \mid \cG) - e^{-(2d -k) \lambda s } \geq  -k\delta.
\end{equation*}
So that by the arbitrariness of $M$ and $\delta$ we have
\begin{equation*}
\forall s > 0,\quad \varlimsup_{t \to \infty}
       \I_{\{Y_t = s\}} \left( P(\xi_t (0) = 0 \mid \cG) - e^{-(2d -k) \lambda Y_t } \right)  \geq  0.
\end{equation*}
In particular if $1- \alpha < \frac{1}{2d+1}$, then for every $s > 0$ we
have
\begin{equation*}
    \varlimsup_{t \to \infty} \I_{\{Y_t = s\}} P(\xi _t(0) =0 \mid \cG) = 1.
\end{equation*}
For the proof of a),
we essentially follow the same structure of the proof of
Theorem~\ref{teo:closeness}.(i). Notice that when $\alpha \in (1/2, 1)$ the
estimates for the probability of events $C_n(\epsilon')$ and $D_n(\epsilon)$
are still available (see Corollary~\ref{coro:bound_Cn} and
Lemma~\ref{lema:bound_Dn}, respectively) for $\epsilon$ and $\epsilon'$ to
be chosen below.
One important difference is that now we have to consider a variation of event
$B_n$ defined in Lemma~\ref{lema:bound_Bn}. The higher value of $\alpha$ will
imply that we expect to have a structure of renewals that is not as extremely
sparse as in the case $\alpha \in (0,1/2)$, but is still sparse nonetheless.
We define $B_n^{k} = B_n^{k}(\epsilon)$ by
\begin{equation*}
B_n^{k}
    := \Bigl\{
        \begin{array}{l}
        \exists \text{ distinct }\{z_j\}_{j=0}^{k+1} \in \sB(n^{3}),\
        s \in [2^{n-1},2^{n+1}]; \\
        \cR_{z_0} \cap [s,s+1] \neq \emptyset,\
        \cR_{z_j} \cap [s,s+2 \cdot 2^{n\epsilon}] \neq \emptyset
        \text{ for $1 \le j \le k+1$}
        \end{array}
    \Bigr\}.
\end{equation*}
Adapting the argument of Lemma~\ref{lema:bound_Bn} shows that for
$\epsilon > 0 $ fixed so that
$1 - \alpha > \frac{1}{k+2} + \epsilon$ we have for some universal $K = K(k)$
\begin{equation*}
    P(B_n^{k}) \le K n^{3d(k+2)} 2^{-n(k+1 - (k+2)\alpha - (k+2)\epsilon)}
\end{equation*}
and $P(B_n^{k})$ is summable on $n$. We choose $\epsilon' < \epsilon$ so
that on event $D_n^{\comp}$ for every $z_0, z_1,z_2,  \ldots, z_{k} \in
\sB(n^3)$ and
every interval $I \subset [2^{n-1}, 2 ^{n+1}]$ of length
$2^{n \epsilon }/2$ there is an interval of length $2^{n \epsilon'}$
with no cure points in $\cup_{i=0}^{k} \cR_{z_i}$.

Thus, taking $\epsilon(k,\alpha)$ small we can ensure that
\begin{equation*}
\smash{\I_{(B_j^{k})^{\comp} \cap D_j^{\comp}}} \to 1
\qquad \text{a.s.\ as $j \to \infty$}.
\end{equation*}
Arguing as in the proof of Theorem~\ref{teo:closeness}(i), taking
$n(t) = \lfloor \log_{2} t \rfloor$ we show that on event
$G_n := (B_{n}^{k})^{\comp} \cap D_{n}^{\comp}$ there will be at most $k+1$ different sites
$z \in \sB(n^{3})$ such that $\cR_z \cap [t-2^{n\epsilon}, t] \neq \emptyset$,
and also some interval $I \subset [t-2^{n\epsilon}, t-2^{n\epsilon}/2]$ with
$\sB(n^{3}) \times I$ without cure marks and satisfying $|I| =2^{n\epsilon'}$.
We define event $C_n = C_n(\epsilon ')$ for this $\epsilon'$. We have,
as in the proof of Theorem \ref{teo:closeness}(i) that
$P(C^{\comp}_n\cap H_{n-1} \mid \cG) \rightarrow 1 \ a.s.$
Given this, we have on event $Y_t \ge 2^{n \epsilon}$ that during interval
$I \subset [t-Y_t, t]$ the origin is infected and then
\begin{equation*}
\smash{\lim_t} \I_{\{Y_t > 2^{n\epsilon}\} \cap G_n}
    \bigl(P(\xi_t(0) = 1 \mid \cG) - 1\bigr)
    =0.
\end{equation*}

Now let us consider  $\{Y_t \le 2^{n \epsilon}\} \cap G_n$.
Here, the origin is one of the $k+1$ sites with renewals on
$\sB(n^{3}) \times [t-2^{n\epsilon},t]$.
There are at most $k$ neighbours of the origin with cure marks on the
interval $[t-2^{n \epsilon}, t]$, so we have that if $C_n\cup H_{n-1}^{\comp}$ does
not occur, then at least $2d-k$ neighbours of the origin have no cure marks
and must be infected during this whole interval.
The probability that none of
these neighbours has a transmission to the origin during $[t-Y_t, t]$
can be controlled as $t \to \infty$, see Remark~\ref{rem:refmademe}
succeeding Corollary~\ref{refmademe}.
If any of these infected neighbours
transmits the infection to the origin then the origin must end up infected at
time $t$. In other words, we have on $\{\tau=\infty\}$ that
\begin{alignat*}{2}
    &\smash{\varliminf_t}\
    \I_{\{Y_t \le 2^{n\epsilon}\} \cap G_n}
    \Bigl(
    P(\xi_t(0) = 1 \mid \cG) - \bigl(1 - e^{-(2d-k) \lambda Y_t}\bigr)
    \Bigr)
    &&\ge 0, \\
    \text{or equivalently,\; }
    & \smash{\varlimsup_t}\
    \I_{\{Y_t \le 2^{n\epsilon}\} \cap G_n}
    \Bigl( P(\xi_t(0) = 0 \mid \cG) - e^{-(2d-k) \lambda Y_t} \Bigr)
    \le 0.
    &&
\end{alignat*}
This finishes the sketch of the proof of Theorem~\ref{teo:closeness_improved}.
\end{proof}

\section{An example}
\label{sec:example}

As anticipated in the Introduction, we now give an example of distribution $\mu$
on $(0,\infty)$ that belongs to the domain of attraction of a stable law of
index one, but for which the associated contact renewal process has
$\lambda_c=0$. Of course, it suffices to consider $d=1$. The question of
whether infinite first moment could be enough for $\lambda_c=0$ remains open
for the moment.

\begin{teo}
\label{teo:example1}
Let $t_0>e$ be fixed, and consider the probability measure $\mu$ on
$(0,\infty)$, given by
\begin{equation}
\label{ex-1}
\mu(t,\infty) = \bar{F}(t):= K L(t)/t, \quad t > t_0,
\end{equation}
where $L(t)=\exp{ (\ln t/\ln{\ln t})}$, and $K$ is the normalizing constant.
If we consider the renewal contact process on $\ZZ$ with
interarrival distribution $\mu$ as above, then $\lambda_c=0$.
\end{teo}

The proof follows the same line of argument as in \cite{FMMV}, identifying
suitable scales for the tunnelling event to happen with positive probability.
Before setting the convenient scales, we recall information about the renewal
process under consideration.

\medskip
\noindent
\textbf{Notation.}
For a renewal process (starting at time zero, say) identified by renewal times
$S_k=T_1+\dots + T_k$, $k \ge 1$, where the random variables $\{T_i\}_i$ are
i.i.d.~with distribution $\mu$, we write $Z_t$ and $Y_t$ for the corresponding
overshooting and age processes:
\begin{equation}\label{lho-1}
Z_t=S_{N_t+1}-t; \; Y_t=t -S_{N_t},
\text{ where $N_t$ is defined by $S_{N_t} \le t < S_{N_t+1}$}.
\end{equation}
Let also $m(t)=\int_0^t \bar F(s)ds$, for $t >0$.
Moreover, when referring to the renewal process attached to site
$j \in \ZZ$ we shall add a superscript $j$ to the corresponding variables.

\begin{remark}
Theorem 6 in \cite{E} implies that if $0<\theta <1$, then
\begin{equation}
\label{erickson}
P\bigl(Z_t >m^{-1} (\theta m(t)) \bigr) \sim  1-\theta \; \text{ as }  t \to \infty.
\end{equation}
\end{remark}

\begin{lema}
\label{calculus}
For the distribution $\mu$ under consideration and $\alpha>0$
one has
\begin{equation}\label{lho}
\lim_{t \to \infty} \frac{m(t/(\ln t)^\alpha)}{m(t)}=e^{-\alpha}
\end{equation}
\end{lema}

\begin{proof}
We first note that we may apply L'H\^opital's rule to the quotient on the left hand side of~\eqref{lho},
thus reducing to looking at the  quotient of derivatives of its respective terms, which results in
\begin{equation}\label{lho+1}
\exp\Bigl[
    \Big\{\frac{\ln t}{\ln(\ln t-\alpha\ln\ln t)}-\frac{\ln t}{\ln\ln
    t}\Big\}-\alpha\Big\{\frac{\ln\ln t}{\ln(\ln t-\alpha\ln\ln t)}\Big\}
    \Bigr]
    \cdot \Bigl(1-\frac{\alpha}{\ln t}\Bigr)
\end{equation}
and the claim of the lemma follows readily by checking that the expressions in braces in~\eqref{lho+1}
converge, respectively, to 0 and 1 as $t\to\infty$.
\end{proof}

For the tunnelling event, let us consider the following (time and space) scales: we set $\alpha>0,\,R_0\geq e$ and $L_0=0$;
then, for $k \ge 0$, let
\begin{equation*}
R_{k+1}=R_{k} + \frac{R_k}{(\ln R_k)^\alpha},
\end{equation*}
\begin{equation*}
L_{k+1}=\min\{j \ge L_{k}+1 \colon Z^j_{R_{k}} > R_{k+1}-R_{k}\}.
\end{equation*}
For convenience, let us write $r_k=R_k-R_{k-1}$ for $k \ge 1$, $r_0=R_0$,
$M_k=\ln r_k$ and $\ell_k=\ln R_k$;
let also $\cI_0=[1,L_1]$, and $\cI_k=[L_k,L_{k+1}]$, $k\geq1$.
The following statement estimates the growth rate of sequence $\ell_k$.

\begin{lema}
	\label{tunn1}
	Let $\beta = (1+\alpha)^{-1}$.
	Then, provided $R_0=R_0(\alpha)$ is large enough, we have that
	$\ell_k \geq (\ell_0 + k)^{\beta}$ for every $k\geq0$.
\end{lema}

\begin{proof}
	We have $R_{k+1} = R_k (1 + \ln^{-\alpha} R_k)$, implying that
	\begin{equation}\label{polinc}
	\ell_k
	= \ell_{k-1} + \ln \bigl(1 + \ell_{k-1}^{-\alpha}\bigr)
	\geq \ell_{k-1} + \beta \ell_{k-1}^{-\alpha}=:g(\ell_{k-1}),
	\end{equation}
	as soon as $R_0$ is large enough.
	
	We argue the claim of the lemma by induction. It clearly holds true for $k=0$, since $\ell_0\geq1$. 
	Assuming it so does for $k=n\geq0$, and using the fact that $g$ is increasing in $[1,\infty)$,
	it follows from~\eqref{polinc} that
	\begin{equation*}
	\ell_{n+1}\geq g(\ell_n)\geq g((\ell_0+n)^\beta)=(\ell_0+n)^\beta+\beta(\ell_0+n)^{-\alpha\beta}\geq (\ell_0+n+1)^\beta,
	\end{equation*}
	and the argument for the induction step is complete as soon as the latter inequality is justified.
	This may be done by writing $(\ell_0+n+1)^\beta-(\ell_0+n)^\beta$, using
    the Mean Value Theorem, as 
	\begin{equation*}
	\frac{\beta}{(\ell_0+n+\zeta)^{1-\beta}}\leq\frac{\beta}{(\ell_0+n)^{1-\beta}}=\frac \beta{(\ell_0+n)^{\alpha\beta}},
	\end{equation*}
	for some $\zeta\in[0,1]$; the latter identity follows from the choice of $\beta$.
\end{proof}

We also need a lower bound on the growth of $m(R_k)$. 
\begin{lema}
	\label{mrk}
	Let $\cL_k=m(R_{k+1})-m(R_k)$. Then, if $R_0$ is large enough, we have that for every $k\geq0$.
	\begin{equation}\label{mrk1}
		\cL_k\geq e^{\sqrt{\ell_k}}
	\end{equation}
\end{lema}

\begin{proof}
	
	For $R_0$ large, we may write $m(R_{k+1})-m(R_k)$ as
	\begin{equation}\label{mrk2}
		K\int_{R_k}^{R_{k+1}}e^{\frac{\ln s}{\ln\ln s}}\,\frac{ds}s=K\int_{\ell_k}^{\ell_{k+1}}e^{\frac s{\ln s}}\,ds
		\geq K(\ell_{k+1}-\ell_{k})\, e^{\frac{\ell_k}{\ln \ell_k}}.
	\end{equation}
    Now, 
    \begin{equation}\label{mrk3}
    	\ell_{k+1} -\ell_{k} = \ln\Big(R_k+\frac{R_k}{\ell_k^\alpha}\Big)-\ln(R_k)=\ln\Big(1+\frac1{\ell_k^\alpha}\Big)\geq\frac1{2\ell_k^\alpha},
    \end{equation}
    and the claim of the lemma follows readily by using the above inequality in~\eqref{mrk2}.
\end{proof}

\begin{proof}[Proof of Theorem \ref{teo:example1}]

Since the probability of no renewals on $\{0\} \times [0,R_0]$ is 
positive for any $R_0$, for the tunnelling it suffices to show that for any value
$\lambda>0$ of the infection rate, we may take $R_0$ so large 
that $\sum_{k \ge 0} P(B_k) <1$, where $B_k$ are events to be shortly defined, 
such that in the complement of
$\cup_{k \ge 0} B_k$  there exists an infection path starting at
$\{0\} \times [0, R_0]$  and continuing forever, see
Figure~\ref{fig:survival_example}.
We stress that in order for this strategy to work, $R_0$ has to be taken suffficiently large
along all of the steps of the argument.
Similarly to Section 4 of \cite{FMMV}, the events $B_k$ are defined as the
union of the following events:

\begin{enumerate}[(I)]
\item $\{L_{k+1} > L_{k} +M_k\}$;

\item For a suitable $V_k$ (as defined below), the rectangle
    $A_k:=\cI_k \times [R_{k}-V_k, R_k]$ is not free of renewal (cure) marks;

\item $(L_k, R_k-V_k)$ does not freely-infect $(L_{k+1}, R_k)$ in
    $A_k$ (see Definition~\ref{defi:freely_infects}).
\end{enumerate}

\begin{figure}
\centering
\begin{tikzpicture}[yscale=.8]
\draw (0,0) -- (10,0) (0,0) -- (0,6);
    \coordinate (r0) at (0,1) node[left ] at (r0) {$R_0$};
    \coordinate (r1) at (0,3) node[left ] at (r1) {$R_1$};
    \coordinate (r2) at (0,5) node[left ] at (r2) {$R_2$};
    \coordinate (l0) at (0,0) node[below] at (l0) {$L_0$};
    \coordinate (l1) at (3,0) node[below] at (l1) {$L_1$};
    \coordinate (l2) at (5,0) node[below] at (l2) {$L_2$};
    \coordinate (l3) at (9,0) node[below] at (l3) {$L_3$};
    \coordinate (v0) at (0,.5);
    \coordinate (v1) at (0,.7);
    \coordinate (v2) at (0,.9);
    % help lines
    \foreach \l in {1,2,3} \draw[dashed] (l\l) -- ++(0,6);
    \foreach \l in {0,1,2} \draw[dashed] (r\l) -- ++(10,0);
    % rectangles A_k
    \fill[opacity=.3] ($(r0) - (v0)$) rectangle ($(l1) + (r0)$);
    \fill[opacity=.3] ($(r1) + (l1) - (v1)$) rectangle ($(l2) + (r1)$);
    \fill[opacity=.3] ($(r2) + (l2) - (v2)$) rectangle ($(l3) + (r2)$);
    % labels V_k
    \foreach \l in {0,1,2}
        \draw[|<->|] (r\l) ++(10,0) -- ++($(0,0) - (v\l)$)
        node[midway, right] {$V_{\l}$};
    % infection path
    \draw[ultra thick, blue]
    (l0) -- (r0)
    ($(l1) + (r0) - (v0)$) -- ($(l1) + (r1)$)
    ($(l2) + (r1) - (v1)$) -- ($(l2) + (r2)$)
    ($(l3) + (r2) - (v2)$) -- ($(l3) + (0,6)$);
    \foreach \x in {0,1,2}{
        \draw[ultra thick, blue]
        ($(\x,0) + (r0) + {.33*\x - 1}*(v0)$) -| ++($(1,0) + .33*(v0)$);};
    \foreach \x in {0,1}{
        \draw[ultra thick, blue]
        ($(\x,0) + (r1) + (l1) + {.33*\x - 1}*(v1)$) -| ++($(1,0) + .33*(v1)$);};
    \foreach \x in {0,...,3}{
        \draw[ultra thick, blue]
        ($(\x,0) + (r2) + (l2) + {.25*\x - 1}*(v2)$) -| ++($(1,0) + .25*(v2)$);};
\end{tikzpicture}
\caption{Construction on Theorem \ref{teo:example1}. On the complement of
    $\cup_{k\geq 0} B_k$, gray regions $A_k$ and intervals
    $\{L_k\} \times [R_{k-1} - V_{k-1}, R_k]$
    are free of cure marks, providing sufficient space for the infection
    from the origin to survive in a straightforward way.}
\label{fig:survival_example}
\end{figure}
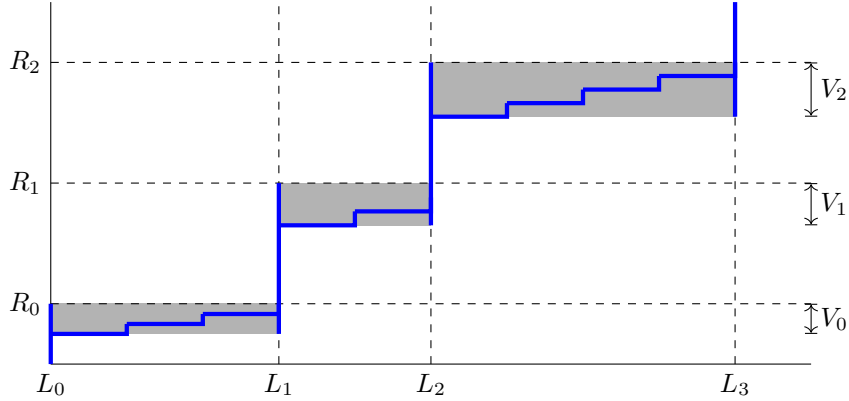

In order for this proof strategy to work we need:
\begin{enumerate}[a)]
\item To control $P (L_{k+1} > L_{k} +M_k)$.

\item To show that for suitable random variables $V_k$ the sum of the
    probabilities of the events $B_k$ as defined above is indeed less than 1.
\end{enumerate}

Using~\eqref{erickson} and Lemma \ref{calculus} we see that for each $k$,
the random variable $L_{k+1}-L_k$ is stochastically dominated by a geometric
distribution with parameter $1-\theta$, where $e^{-\alpha} <\theta <1$. Thus,
\begin{equation}
\label{I}
    P(L_{k+1}-L_k > M_k) < \theta^{M_k}.
\end{equation}

As natural candidate for $V_k$ we have
$V_k=\min\{r_k, Y^{L_k+1}_{R_k},\dots,Y^{L_{k+1}}_{R_k}\}$
which we shall explore when $L_{k+1} \leq L_k +M_k$.

Note that if $[a,a+M] \times [s,s+V]$ is a space-time interval free of cure
marks and such that site $a$ is infected at time $s$, then the probability that
the infection does not reach the space time point $(a+M,s+V)$ is bounded by
that of $G(M,\lambda) > V$, where $G(M, \lambda)$ has distribution Gamma with
parameters $M$ and $\lambda$. Indeed, the rightmost infection path will simply
move as a Poisson process with rate $\lambda$.
The result follows if we can prove that, for suitable $R_0$, the sum over
$k \ge 0$ of the probabilities in \eqref{I} and those in \eqref{II} below are
less than one,
\begin{equation}
\label{II}
P\Bigl(\min\{r_k, Y^{L_k+1}_{R_k},\dots,Y^{L_k+M_k}_{R_k}\} < G(M_k,\lambda)\Bigr)
\end{equation}
with $G(M_k, \lambda)$ as above, independent of the renewal processes.
The probability in \eqref{II} is easily seen to be bounded from above by
\begin{align}
    P(G(M_k, \lambda) &> r_k) + M_k P(G(M_k,\lambda) > Y_{R_k}) \nonumber\\
\label{IIa}
    &\leq M_k e^{-\frac{\lambda}{M_k} r_k} +
        {M_k}^2 E(e^{-\frac{\lambda}{M_k} Y_{R_k}}).
\end{align}
For the second summand on the r.h.s. of~\eqref{IIa},
we write it in terms of the renewal measure $U$ for $\mu$:
\begin{align*}
E(e^{-\frac{\lambda}{M_k} Y_{R_k}})
    &=  \int_0^{R_k} U(ds) \bar F(R_k-s) e^{-\frac{\lambda}{M_k}(R_k-s)} \\
	&\le e^{-\lambda M_k} + U(R_k) - U(R_k -{M_k}^2) \\
	&=  e^{-\lambda M_k} +
        \sum_{i=1}^{{M_k}^2} \big[U(R_k-{M_k}^2+i) - U(R_k -{M_k}^2+i-1)\big].
\end{align*}
Using now Lemma 10 (b) in \cite{E} the last sum is bounded from above by
$\frac{2{M_k}^2}{m(R_k)}$ so that
the second term in the last line of \eqref{IIa} is bounded from above by
\begin{equation}
\label{IIaa}
{M_k}^2e^{-\lambda M_k} +  \frac{2{M_k}^4}{m(R_k)}
    \leq {M_k}^2e^{-\lambda M_k} + \frac{2{M_k}^4}{\cL_{k-1}},
\end{equation}
since $m(R_k) = \cL_{k-1} + m(R_{k-1}) \ge \cL_{k-1}$.
Recall that $M_k = \ln \frac{R_{k-1}}{(\ln R_{k-1})^{\alpha}}$ and $\ell_k = \ln
R_k$. If we choose a sufficiently large $R_0(\alpha)$ then
\begin{equation*}
    \frac{1}{2} \ell_{k-1}
    \le M_k
    \le \ell_{k-1},
    \qquad \text{for every $k \ge 1$}.
\end{equation*}
Hence, putting the inequalities above together and using Lemma~\ref{mrk}, we
get
\begin{equation*}
\sum_{k\geq 0} P(B_k)
    \leq \sum_{k\geq 0} \Bigl(
        \theta^{M_k} + M_k e^{-\frac{\lambda}{M_k} e^{M_k}} +
        {M_k}^2e^{-\lambda M_k} + 2{M_k}^4 e^{- M^{1/2}_{k}}
    \Bigr)
\end{equation*}
Finally, we notice that the estimate $\ell_k \ge (\ell_0+k)^{\beta}$ of
Lemma~\ref{tunn1} ensures that $M_k$ grows fast enough so that the series above
converges. Moreover, given $\epsilon >0$ we may take $\bar{R}(\epsilon)$ so that
$\sum_{k\geq 0} P(B_k) < \epsilon$ if $R_0 >\bar R(\epsilon)$.
\end{proof}

\section*{Acknowlegments}
\label{sec:acknowlegments}

The authors would like to thank the referee for a careful reading of our manuscript.

\section*{Funding}
\label{sec:funding}

LRF was partially supported by CNPq grant 307884/2019-8 and FAPESP grant 2017/10555-0.
A visit of TSM to UFRJ in 2018 had partial support of a Faperj grant E-26/203.048/2016.
DU was partially supported by FAPESP grant 2020/05555-4.
MEV was partially supported by CNPq grant 310734/2021-5 and FAPERJ CNE grant E-26/202.636/2019.

\bibliography{rcp3-spa}

\end{document}